\documentclass[12pt]{amsart}
 
\usepackage{tikz}
\usetikzlibrary{cd}
\usetikzlibrary{decorations.pathreplacing,decorations.markings}

\usepackage{fullpage}
\usepackage{amsfonts,amssymb}

\usepackage{pinlabel,url}
\usepackage[all]{xy}
\usepackage{tikz}
\usepackage{graphicx}
\usepackage{xcolor}
\usepackage{hyperref}
\usepackage{caption}

\usepackage{subcaption}

\usepackage{wrapfig}

\usepackage{float}

\usepackage{thm-restate}

\usepackage[nameinlink]{cleveref}

\newtheorem{theorem}{Theorem}[section]

\newtheorem{lemma}[theorem]{Lemma}
\newtheorem{proposition}[theorem]{Proposition}
\newtheorem{corollary}[theorem]{Corollary}
\newtheorem{claim}[theorem]{Claim}

\theoremstyle{definition}

\newtheorem{example}[theorem]{Example}
\newtheorem{nonexample}[theorem]{Non-example}
\newtheorem{definition}[theorem]{Definition}
\newtheorem{question}[theorem]{Question}

\numberwithin{equation}{section}

\newcommand{\Ends}{{\rm Ends}}
\newcommand{\Vol}{{\rm Vol}}
\newcommand{\Map}{{\rm Map}}

\newcommand{\intt}{{\rm{Int}}}

\newcommand{\diam}{{\rm diam}}

\newcommand{\Area}{{\rm Area}}

\newcounter{mcomments}

\title[Volumes of end-periodic mapping tori]{Volumes of end-periodic mapping tori}

\author[M. Loving]{Marissa Loving}

\date{\today}

\begin{document}


\begin{abstract}
    In this expository paper, we provide an intuition and illustration-driven overview of two recent results (the main theorems of \cite{EndPeriodic1} and \cite{EndPeriodic2}) that tie the dynamics of certain homeomorphisms of infinite-type surfaces, called \emph{end-periodic homeomorphisms}, to the geometry of their associated (compactified) mapping tori. These results are analogues of a theorem of Brock in the finite-type setting for mapping tori of pseudo-Anosov homeomorphisms \cite{Brock-mappingtorus-vol}. 
\end{abstract}

\maketitle

\setcounter{tocdepth}{1}
\tableofcontents

\section{Preface}

This expository paper builds on lecture notes from a minicourse entitled, \emph{Infinite-type surfaces, end-periodic homeomorphisms, and the geometry of 3-manifolds}, which was given by the author as part of the 2023 Riverside Workshop on Geometric Group Theory. That minicourse and this work both focus primarily on the author's recent results on the volumes of end-periodic mapping tori, which are joint with Elizabeth Field, Heejoung Kim, and Christopher Leininger \cite{EndPeriodic1} and Elizabeth Field, Autumn Kent, and Christopher Leininger \cite{EndPeriodic2}. We will lean much more heavily on intuition and illustrations than is typical in mathematical writing with an effort to highlight key ideas and provide helpful examples rather than to fully flesh out proofs that can already be found in complete detail in the relevant source material. 

We begin in \Cref{sec:intro} with an introduction that sketches some of the relevant background on infinite-type surfaces and foliations of $3$-manifolds. In \Cref{sec:end-periodic-mapping-tori}, we introduce the star of the show, end-periodic homeomorphisms of infinite-type surfaces, and discuss the geometry and structure of their associated (compactified) mapping tori. The final two sections are devoted to sketching the bounds on volumes of end-periodic mapping tori with \Cref{sec:lower-bound} focusing on the lower bound of \cite{EndPeriodic2} and \Cref{sec:upper-bound} focusing on the upper bound of \cite{EndPeriodic1}.

\subsection*{Acknowledgments} My perspective on these topics have been shaped significantly by my collaborations with Chris Leininger and Elizabeth Field as well as a series of talks by (and accompanying conversations with) Yair Minsky at the University of Wisconsin-Madison in the Spring of 2023 on the topic of foliations and the geometry of 3-manifolds and his collection of work with Michael Landry and Sam Taylor \cite{LandryMinskyTaylor2021, LandryMinskyTaylor2023}. An earlier draft of this paper was improved by a careful reading and corrections from Connie On Yu Hui, Mujie Wang, and Brandis Whitfield. In particular, \Cref{fig:PeriodicCurve} was suggested by Brandis Whitfield. The exposition has also been corrected and clarified by detailed comments from both anonymous referees for which I am very grateful. Finally, I would like to give my warmest thanks to the organizers of the Riverside Workshop on Geometric Group Theory, Matt Durham and Thomas Koberda, for their hospitality, support, and invitation to contribute to these proceedings.  

\section{Introduction} \label{sec:intro}

\subsection{Infinite-type surfaces}

We will discuss surfaces of both finite and infinite type where a \emph{finite-type surface} is a surface with finitely-generated fundamental group and an \emph{infinite-type surface} is a surface whose fundamental group is not finitely generated. Throughout we will use $\Sigma$ to denote a finite-type surface and $S$ to denote an infinite-type surface. In both cases we require that our surfaces be connected and orientable. A classical result tells us that the homeomorphism type of a finite-type surface is determined entirely by its genus, number of boundary components, and number of punctures (or marked points). There is a similar classification result for infinite-type surfaces without boundary given by work of Ker\'ekj\'art\'o \cite{Ker} and Richards \cite{Richards}. Note that the genus of an infinite-type surface can be either finite or countably infinite as illustrated in \Cref{fig:Cantor}. See also \Cref{fig:ClassificationTheorem} for an example of three representations of a surface which are not ``obviously" homeomorphic. 

\begin{figure}
    \centering
    \begin{subfigure}[b]{0.45\textwidth}
        \centering
        \includegraphics[width = 3in]{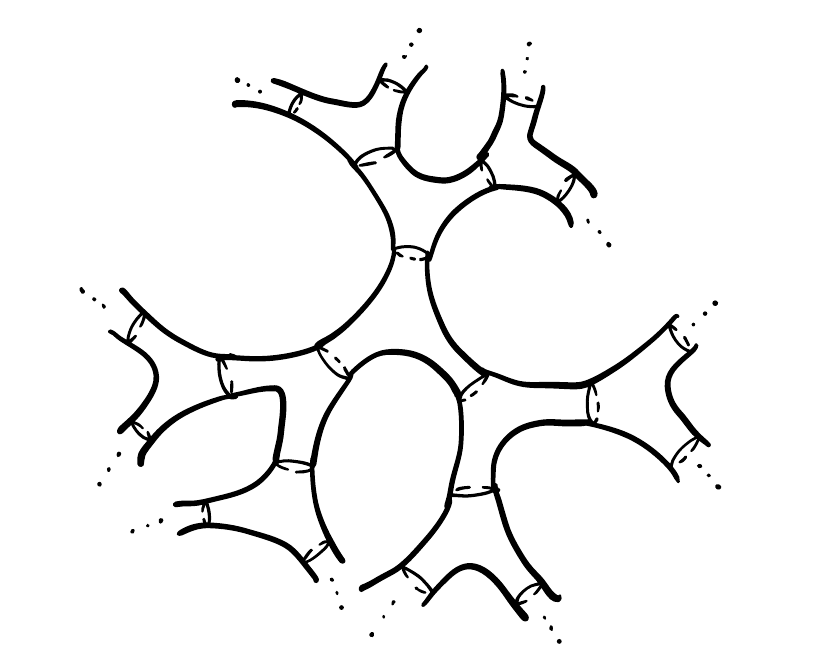}
        \caption{The Cantor Tree surface which is homeomorphic to a sphere minus a Cantor set of punctures.}
        \label{fig:CantorTree}
    \end{subfigure}
    \hfill
    \begin{subfigure}[b]{0.45\textwidth}
        \centering
        \includegraphics[width = 3in]{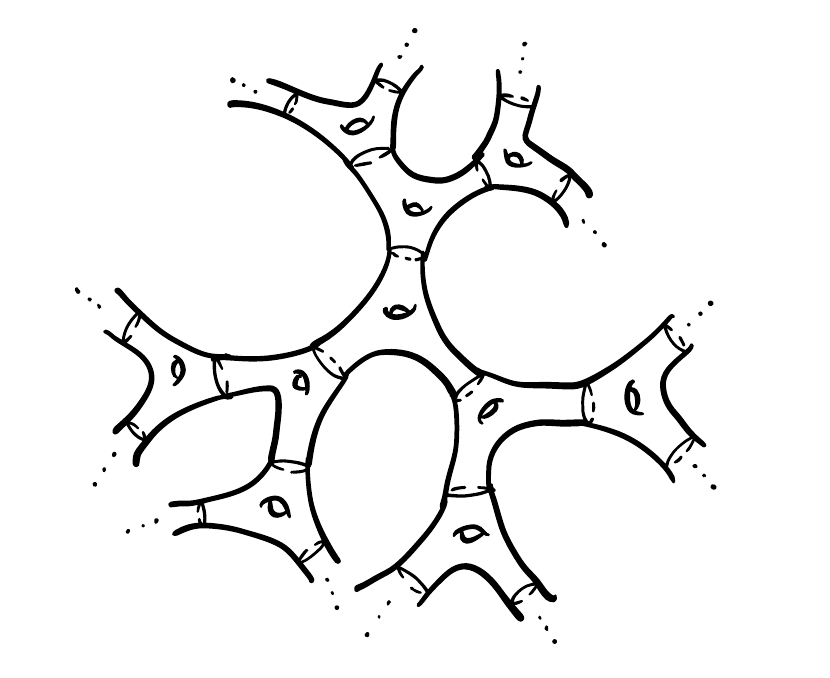}
        \caption{The Blooming Cantor Tree surface which is homeomorphic to a sphere minus a Cantor set of punctures where each end is also accumulated by genus.}
        \label{fig:BloomingCantorTree}
    \end{subfigure}
    \caption{Two infinite-type surfaces with end spaces homeomorphic to a Cantor set.}
    \label{fig:Cantor}
\end{figure}

\begin{theorem}[Classification of infinite-type surfaces]
    A surface is determined, up to homeomorphism, by its genus, its space of ends, and the subspace of ends accumulated by genus. 
\end{theorem}

Intuitively, ends are a way of ``going to infinity" in a surface. Precisely, ends are equivalence classes of ``exiting sequences". Here an \emph{exiting sequence} is a sequence $\{U_n\}_{n \in \mathbb N}$ of connected open subsets of a surface $S$ such that: 

\begin{enumerate}
    \item[(1)] $U_n \subset U_m$ whenever $m < n$,
    \item[(2)] $U_n$ is not relatively compact for any $n \in \mathbb N$,
    \item[(3)] $U_n$ has compact boundary for all $n \in \mathbb N$, and 
    \item[(4)] any relatively compact subset of $S$ is disjoint from all but finitely many $U_n$'s. 
\end{enumerate}

See \Cref{fig:ExitingSequence}. Two exiting sequences are \emph{equivalent} if every element of the first is eventually contained in some element of the second, and vice versa. The collection of all ends of an infinite-type surface $S$ is denoted \Ends(S) and can be topologized as follows: given an open set $U$ of $S$ with compact boundary, let $U^*$ be the set of points in \Ends(S) of the form $p = [\{V_m\}]$ with $V_m \subset U$ for some $m$. The collection of such $U^*$ form a basis for the topology on \Ends(S).


\begin{figure}
    \centering
    \includegraphics[scale=.6]{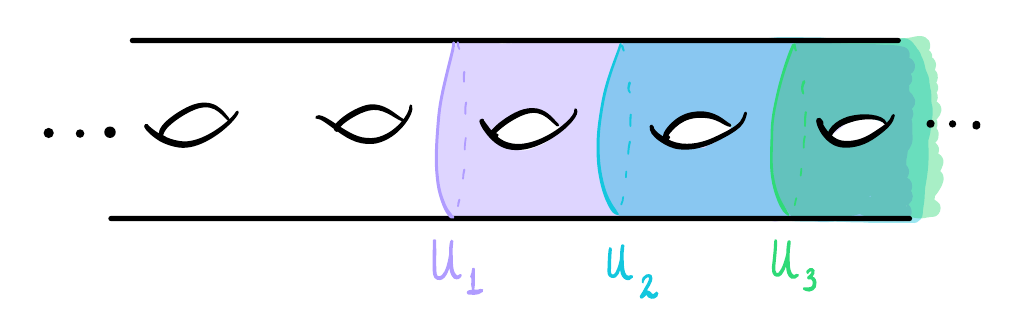}
    \caption{An exiting sequence on the ladder surface.}
    \label{fig:ExitingSequence}
\end{figure}

\begin{theorem}[Richards, \cite{Richards}]
    For any connected surface $S$, the space \Ends(S) is totally disconnected, second countable, and compact. In particular, \Ends(S) is homeomorphic to a closed subset of the Cantor set.
\end{theorem}

\begin{figure}
    \centering
    \includegraphics[width=.9\textwidth]{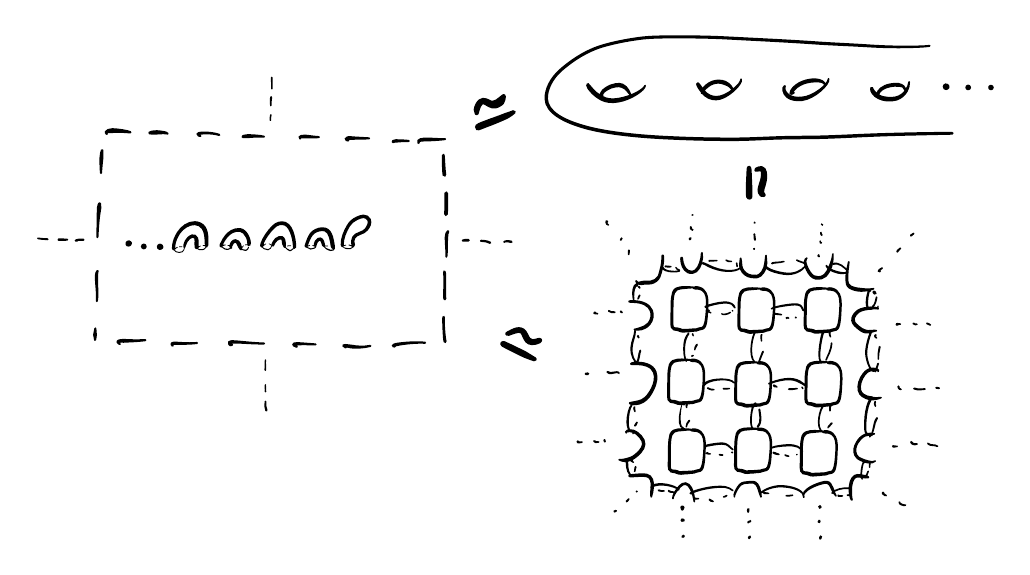}
    \caption{It is a consequence of the classification theorem that these three surfaces are homeomorphic.}
    \label{fig:ClassificationTheorem}
\end{figure}

The interested reader can refer to a survey of big mapping class groups by Aramayona and Vlamis for a more in-depth discussion of infinite-type surfaces \cite{AramayonaVlamis2020}.

\subsection{Homeomorphisms of infinite-type surfaces}

The most straightforward examples of homeomorphisms of infinite-type surfaces are those that directly correspond to homeomorphisms of finite-type surfaces such as rotations, Dehn twists, or, indeed, any homeomorphism that is compactly supported. However, there also exist examples of \emph{intrinsically infinite-type homeomorphisms} (this terminology was coined in \cite{AbbottMillerPatel}), i.e. homeomorphisms that are not compactly supported or a limit of compactly supported homeomorphisms. 


\begin{figure}
    \centering
    \includegraphics[scale = .5]{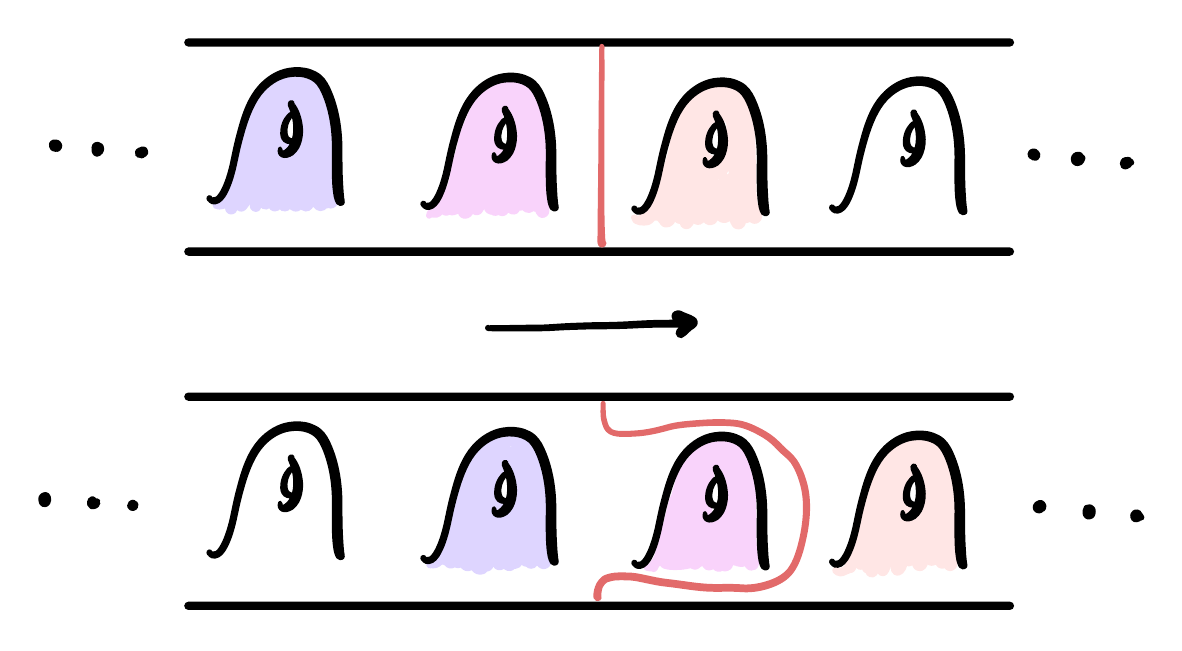}
    \caption{A handle strip and its image after applying the associated shift map.}
    \label{fig:HandleStrip}
\end{figure}

The example which we will make heavy use of throughout this survey is a \emph{handle shift} which can be constructed as follows: Consider a bi-infinite strip $T  \cong [0, 1] \times \mathbb R$ with handles attached at integer intervals. Define a homeomorphism $h: T \to T$ which shifts every handle over by one and which tapers to the identity on each boundary component of $T$. This \emph{handle strip} and its associated shift map is illustrated in \Cref{fig:HandleStrip}. Note that $T$ can be embedded into any infinite-type surface $S$ with countably infinite genus and $h$ can be extended by the identity to a homeomorphism of $S$. The resulting homeomorphism of $S$, which we will also call $h$, is a \emph{handle shift}. Note that the handle shift $h: S \to S$ is only of intrinsically infinite type when $T$ is embedded between two distinct ends of $S$.


Just as in the finite-type setting, we can study the mapping class group of an infinite-type surface $S$, \Map(S), which is the group of orientation-preserving homeomorphisms of $S$ up to isotopy. This is often given the moniker the ``big mapping class group", which is apt since it is a non-locally compact Polish group, and is an interesting object of study in its own right! However, in this work we will keep our focus relatively narrow and will not explore big mapping class groups very deeply. For a broad survey of big mapping class groups see \cite{AramayonaVlamis2020}. 

\subsection{Foliations of 3-manifolds}

A \emph{foliation} of an $n$-dimensional manifold $M$ is a decomposition of $M$ into $(n-1)$-dimensional submanifolds (possibly immersed). Thus, a foliation of a surface is a decomposition into lines or closed curves as illustrated in \Cref{fig:FoliationsSurfaces} while a foliation of a 3-manifold is a decomposition into surfaces.

\begin{figure}
    \centering
    \includegraphics[width=.9\textwidth]{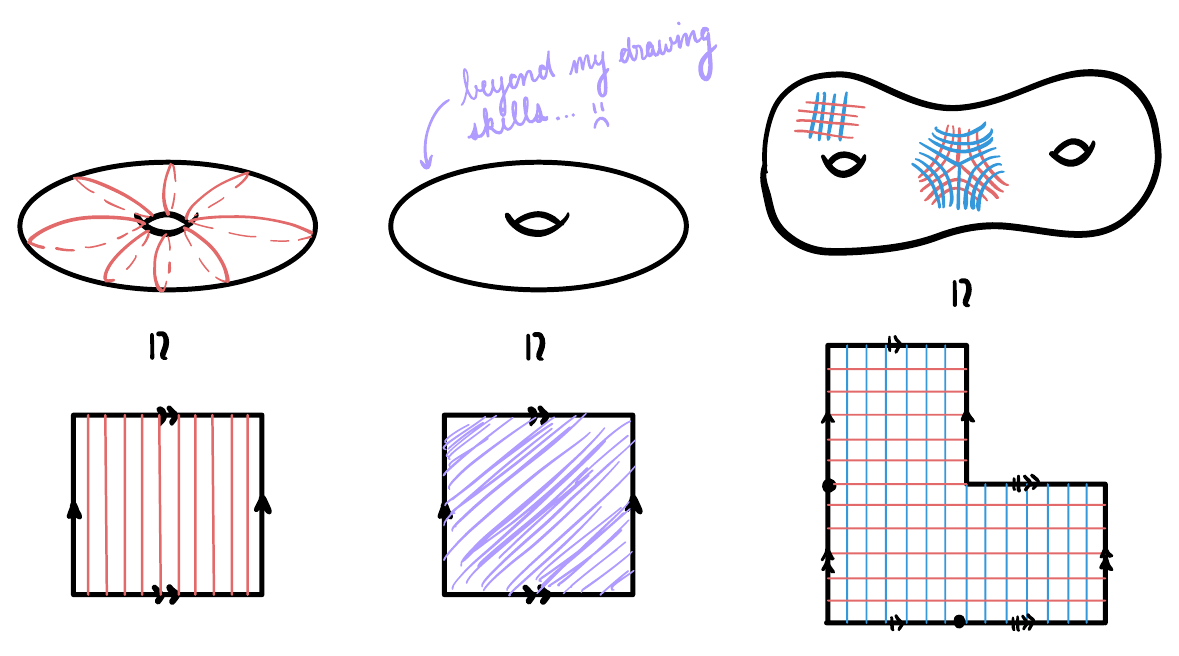}
    \caption{A foliation of a torus by simple closed curves, a foliation of the torus coming from a line of irrational slope (note that it has no closed leaf), and a pair of singular, transverse foliations on a genus 2 surface.}
    \label{fig:FoliationsSurfaces}
\end{figure}

\begin{example}
    An irreducible fibered $3$-manifold admits a foliation with leaves given by the surface fibers. Note that every irreducible fibered $3$-manifold arises as the mapping torus $M_f$ of some element $f \in \Map(\Sigma)$. See \Cref{fig:MappingTorus} for an illustration and \cite{StallingsFiber} for further details.
\end{example}

\begin{figure}
    \centering
    \includegraphics[width = \textwidth]{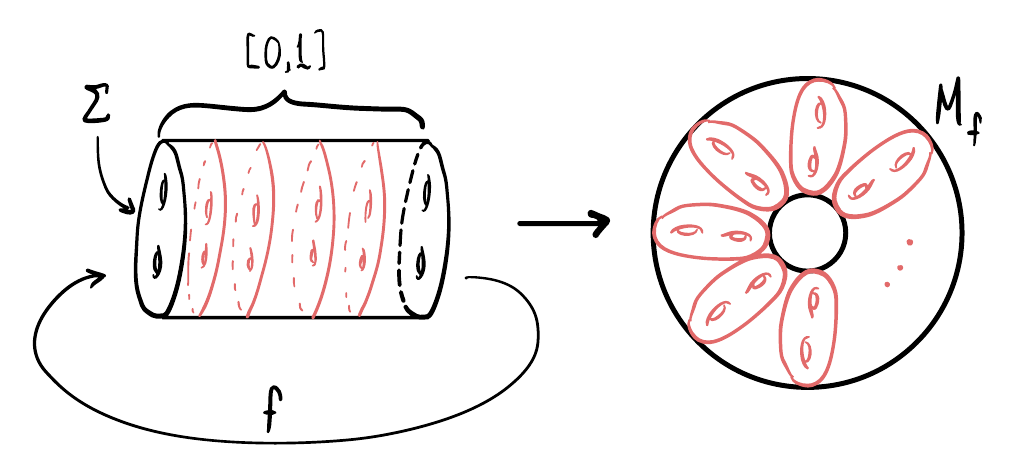}
    \caption{The mapping torus of a homeomorphism $f: \Sigma \to \Sigma$. Note that $M_f$ only depends on the mapping class of $f$.}
    \label{fig:MappingTorus}
\end{figure}

Not only do mapping tori give us a simple example of foliated $3$-manifolds, they also play an important role in understanding the geometry of $3$-manifolds. Recall the following foundational result of Thurston.

\begin{theorem}[Thurston's Hyperbolization Theorem, \cite{Thurston.Geometrization.1}] \label{t:hyperbolization}
    Let $M$ be an irreducible and atoroidal $3$-manifold. If $M$ is Haken, then $M$ admits a hyperbolic metric. 
\end{theorem}

When $M$ is fibered, and hence a mapping torus, the statement of hyperbolization reduces to the following characterization. For further discussion and details see \cite{TNotes, Otal-Survey-Haken, Otal-Survey-Fibered}.

\begin{theorem}[Thurston, \cite{Thurston.Geometrization.2}] \label{t:mapping-tori}
    Let $f \in \Map(\Sigma)$. $M_f$ is hyperbolic if and only if $f$ is psuedo-Anosov. 
\end{theorem}

In this paper, we will often focus on the following guiding question.

\begin{question}
    How do the dynamics of $f$ predict/determine the geometry of $M_f$?
\end{question}

However, an even more pressing question in the reader's mind might be, how do these foliations relate to the infinite-type surfaces discussed at the beginning of this section? The following result of Cantwell--Conlon hints at the connection. Here by \emph{open} we mean without boundary. 

\begin{theorem}[Cantwell--Conlon, \cite{CantwellConlon1987}]
    Every (open) surface is a leaf (of a foliation of a closed, connected $3$-manifold). 
\end{theorem}

Thus, there are necessarily much more complicated foliations of $3$-manifolds than the single example we have given so far. An important type of foliation that we will consider is a \emph{depth-one foliation}. This is a foliation with two types of leaves, compact (depth zero) and non-compact (depth one) leaves, where all the non-compact leaves accumulate onto the compact leaves. An attempt at an illustration of this is shown in \Cref{fig:DepthOneFoliation}, but perhaps a more meaningful picture can be seen if we go down a dimension as shown in \Cref{fig:ClosedLeafSurface}. 

\begin{figure}
    \centering
    \includegraphics[width = .5\textwidth]{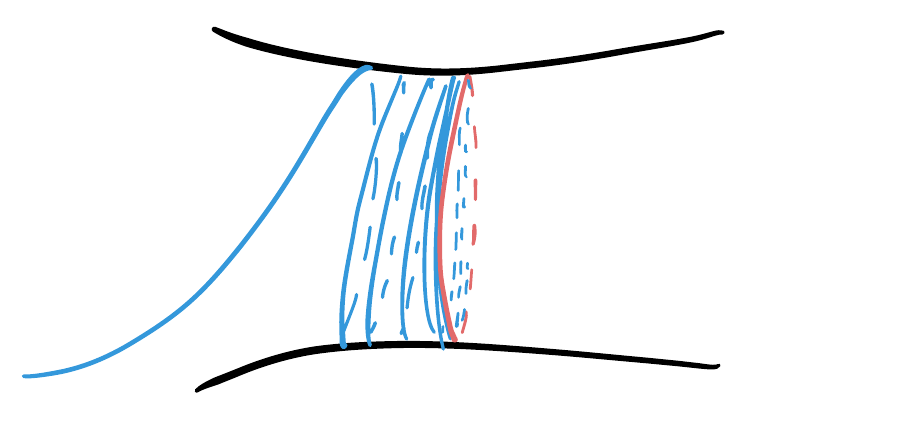}
    \caption{An infinite geodesic accumulating onto a simple closed geodesic on a surface.}
    \label{fig:ClosedLeafSurface}
\end{figure}

Note that for a closed $3$-manifold, a compact leaf has finite genus without boundary components and a non-compact leaf has infinite genus without planar ends or boundary components. 

Why are depth-one foliations so special? Consider a closed 3-manifold $M$ with a \emph{taut} depth one foliation $\mathcal F$. Note that we call a foliation \emph{taut} if there exists a homotopically non-trivial simple closed curve transverse to every leaf. Let $\mathcal F_c \subset \mathcal F$ denote the collection of compact leaves of $\mathcal F$. Then $M - \mathcal F_c$ is a collection of fibered $3$-manifolds with \emph{end-periodic monodromy}! We will postpone our discussion of end-periodic homeomorphisms to the next section. In the meantime we pose the following natural question which the author first heard stated by Yair Minsky in a talk he gave in Spring 2023.

\begin{figure}
    \centering
    \includegraphics[width = .9\textwidth]{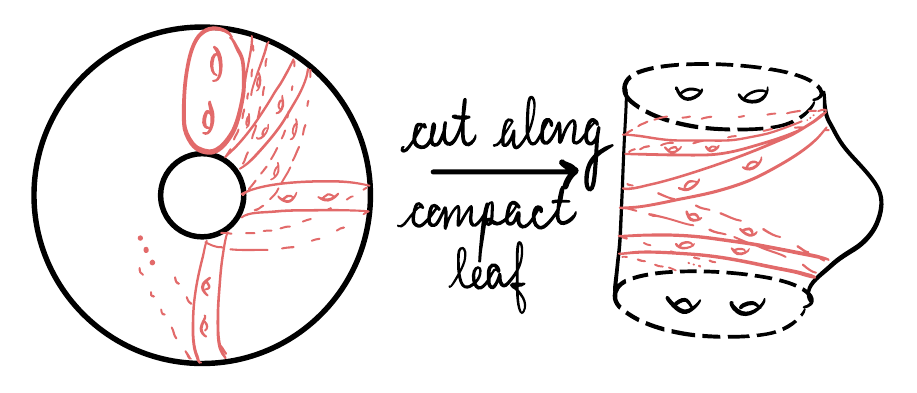}
    \caption{A depth one foliation of a closed $3$-manifold with a single, compact genus 2 leaf.}
    \label{fig:DepthOneFoliation}
\end{figure}

\begin{restatable}[Minsky]{question}{TautGeometry} \label{q:taut-geometry}
    Given a $3$-manifold with a (taut) foliation, can we predict its geometry?
\end{restatable}


We will not mention tautness very often. However, we will make use of the following ``normal form" for surfaces in taut foliated $3$-manifolds. See \cite{Rou74, CandConII, Thu72, Has86}.

\begin{theorem}[Roussarie--Thurston Normal Form] \label{thm:RoussarieThurston}
    Suppose $\mathcal F$ is a transversely oriented, taut foliation of a $3$-manifold $M$ for which $\partial M$ is a union of leaves. Suppose $\Sigma \subset M$ is a properly embedded, incompressible torus or annulus. After an isotopy of the embedding, which is the identity on the boundary in the annulus case, we have:
    \begin{enumerate}
        \item[(1)] $\Sigma$ is a torus and is either transverse to $\mathcal F$ or is a leaf of $\mathcal F$; or
        \item[(2)] $\Sigma$ is an annulus and is transverse to $\mathcal F$, except at finitely many circle tangencies occuring in the interior of $M$.
    \end{enumerate}
\end{theorem}


\begin{figure}
    \centering
    \includegraphics[scale = .4]{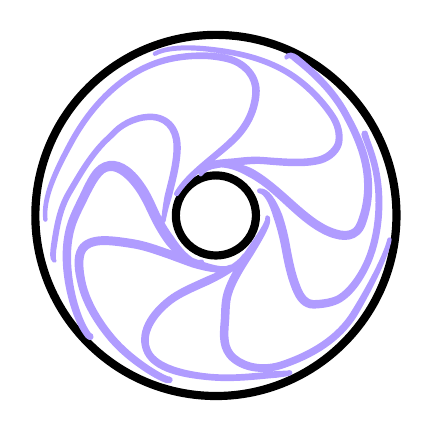}
    \caption{A cross section of the Reeb foliation of the solid torus.}
    \label{fig:ReebFoliation}
\end{figure}

\begin{nonexample}
    The Reeb foliation (see \Cref{fig:ReebFoliation}) of the solid torus is not taut.
\end{nonexample}

It turns out that in the hyperbolic setting, this non-example is the only obstruction to being taut! 


\begin{corollary}[Novikov, \cite{Novikov}] \label{cor:Novikov}
    If $M$ is hyperbolic, $\mathcal F$ has no Reeb component (this is called \emph{Reebless}) if and only if it is taut.
\end{corollary}

\Cref{cor:Novikov} follows quite directly from the following characterization of Reebless foliations.

\begin{theorem}[Novikov, \cite{Novikov}] \label{thm:Novikov}
    Let $M$ be a $3$-manifold and let $\mathcal F$ be a Reebless foliation. Then
    \begin{enumerate}
        \item[(1)] every leaf $\lambda$ is $\pi_1$-injective; and
        \item[(2)] every transverse loop $\gamma \subset M$ is essential in $\pi_1(M)$.
    \end{enumerate}
\end{theorem}

\begin{proof}[Proof of Corollary]
    \Cref{thm:Novikov} tells us that every leaf is essential if $\mathcal F$ is Reebless. Since leaves which do not admit a transversal are necessarily tori \cite{Goodman1975}, we know that if $M$ admits a foliation which is Reebless but not taut, then $M$ is toroidal. Thus, because $M$ is toroidal, it is not hyperbolic. 
\end{proof}
 
\section{End-periodic mapping tori} \label{sec:end-periodic-mapping-tori} 

In the last section we introduced infinite-type surfaces and their homeomorphisms as well as foliations of 3-manifolds. In this section we will introduce the infinite-type surface homeomorphisms which tie these two topics together for us.

\subsection{End-periodic homeomorphisms}


The study of end-periodic homeomorphisms was initiated in unpublished work of Handel--Miller and developed in work of Cantwell, Conlon, and Fenley \cite{CC-examples, CC-book, FenleyThesis1989, FenleyCT, Fenley-depth-one}.

\begin{definition}
    An \emph{end-periodic homeomorphism} $f: S \to S$ of a surface $S$ with finitely many ends all accumulated by genus, as illustrated in \Cref{fig:FiniteEnds}, is a homeomorphism satisfying the following. There exists $m>0$ such that for each end $E$ of $S$, there is a neighborhood $U_E$ of $E$ so that either 
    \begin{enumerate}
        \item[(i)] $f^m(U_E) \subsetneq U_E$ and the sets ${f^{nm}(U_E)}_{n > 0}$ form a neighborhood basis of $E$; or
        \item[(ii)] $f^{-m}(U_E) \subsetneq U_E$ and the sets ${f^{-nm}(U_E)}_{n > 0}$ form a neighborhood basis of $E$.
    \end{enumerate}
    Note that we call such a neighborhood $U_E$ a \emph{nesting neighborhood}.
\end{definition}

\begin{figure}
    \centering
    \begin{subfigure}[b]{0.45\textwidth}
        \centering
        \includegraphics[width = 3in]{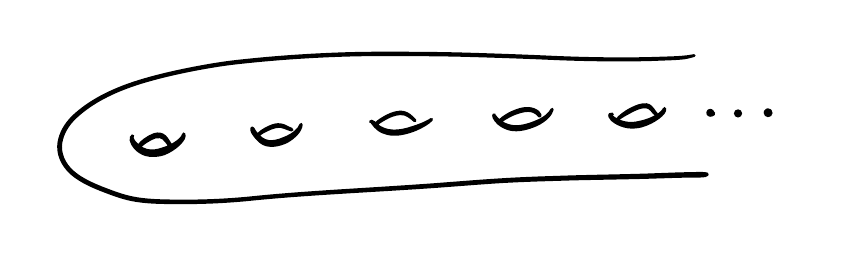}
        \caption{The Loch Ness Monster surface.}
        \label{fig:LochNess}
    \end{subfigure}
    \hfill
    \begin{subfigure}[b]{0.45\textwidth}
        \centering
        \includegraphics[width = 3in]{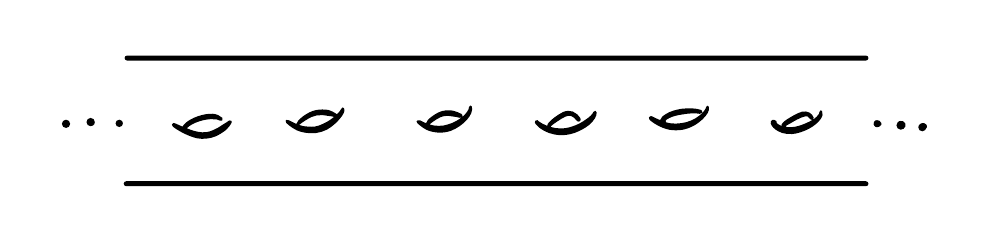}
        \caption{The Ladder surface.}
        \label{fig:Ladder}
    \end{subfigure}
    
    \begin{subfigure}[b]{0.45\textwidth}
        \centering
        \includegraphics[width = 2in]{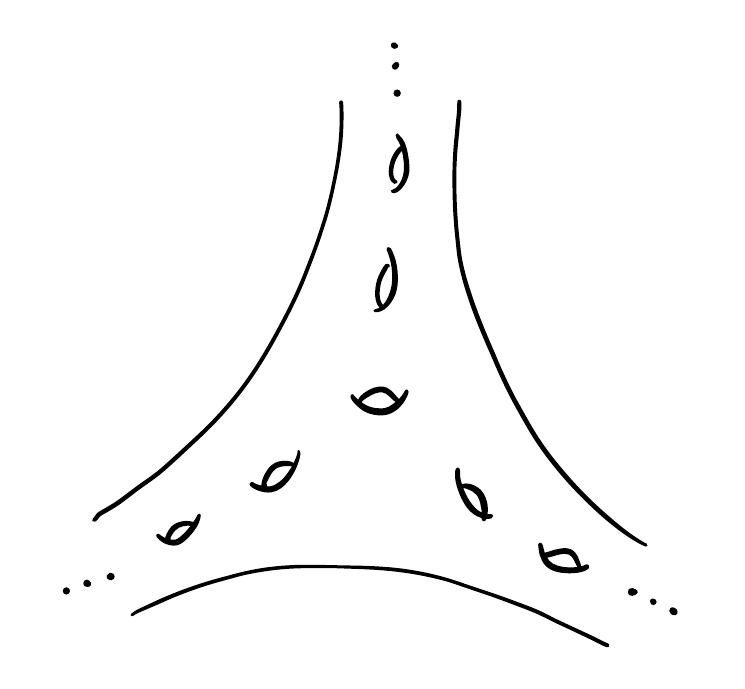}
        \caption{The Tripod surface.}
        \label{fig:Tripod}
    \end{subfigure}
\caption{Three surfaces with finitely many ends all accumulated by genus.}
    \label{fig:FiniteEnds}
\end{figure}

In the above definition, an end of $S$ is called \emph{attracting} if it falls into case (i) and is called \emph{repelling} if it falls into case (ii). Thus, a more concise definition is that a homeomorphism $f: S \to S$ of a surface $S$ with finitely many ends all accumulated by genus is end-periodic if every end of $S$ is either attracting or repelling under the action of $f$. This \emph{periodic} behavior on the ends is illustrated in \Cref{fig:PeriodicInEnds}.

\begin{figure}
    \centering
    \includegraphics[width=.8\textwidth]{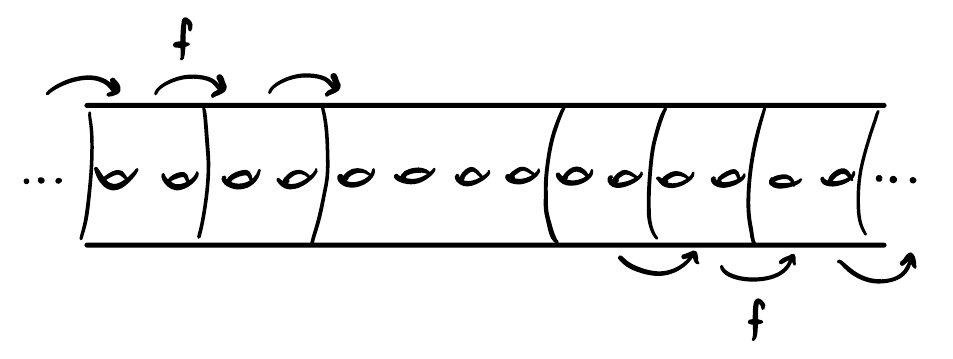}
    \caption{The name \emph{end-periodic} describes the fact that outside of some compact subsurface the action of $f: S \to S$ is periodic in certain neighborhoods of the ends of $S$.}
    \label{fig:PeriodicInEnds}
\end{figure}

\subsection{Geometry of end-periodic mapping tori} 

In this section we pose and then address the following question: what is the geometry of $M_f$ for $f$ end-periodic? It turns out that we will answer a slightly different question where we consider a compactification $\overline M_f$ of $M_f$. The following result appears as Proposition~3.1 in \cite{EndPeriodic1}.

\begin{theorem}\label{t:end-periodic-mapping-tori}
    Let $f: S \to S$ be an end-periodic homeomorphism. Then $M_f$ is the interior of a compact, irreducible $3$-manifold with incompressible boundary called $\overline M_f$. Furthermore,
    
    \begin{enumerate}
        \item[(i)] if $\overline M_f$ is atoroidal, then it admits a convex hyperbolic metric; and 
        \item[(ii)] if $\overline M_f$ is additionally acylindrical, then it admits a unique (up to isometry) convex hyperbolic metric with respect to which $\partial M_f$ is totally geodesic.  
    \end{enumerate} 
\end{theorem}

The first part of \Cref{t:end-periodic-mapping-tori}, establishing the existence of $\overline M_f$, is well-known (see \cite{Fenley-depth-one}), while \Cref{t:end-periodic-mapping-tori} (i) follows from \Cref{t:hyperbolization} and \Cref{t:end-periodic-mapping-tori} (ii) follows from the following straightforward argument. If $\overline M_f$ is atoroidal and acylindrical, then its double $D\overline M_f$ is closed, Haken, \emph{and} atoroidal. A cartoon of this doubling procedure is shown in \Cref{fig:Double}. Thus, by \Cref{t:hyperbolization} and Mostow rigidity $D \overline M_f$ admits a unique hyperbolic metric. This implies that there is a unique convex hyperbolic metric on $\overline M_f$ such that $\partial \overline M_f$ is totally geodesic. We also note that by work of Storm \cite{Storm2} this is also the unique minimal-volume hyperbolic metric on $\overline M_f$. We will often refer to $\overline M_f$ itself as the \emph{end-periodic mapping torus} of $f$.

\begin{figure}
    \centering
    \includegraphics[width=\textwidth]{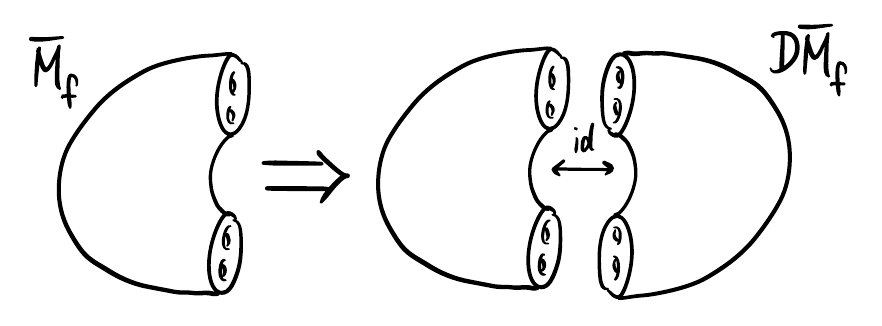}
    \caption{The end-periodic mapping tori $\overline M_f$ and its double $D \overline M_f$,}
    \label{fig:Double}
\end{figure}

In light of this result, it is natural to ask the following.

\begin{question}\label{q:irreducibility}
    What conditions on the end-periodic homeomorphism $f$ guarantee that $\overline M_f$ is either atoroidal or acylindrical?
\end{question}

Recall that in the finite-type setting we know that $M_{f}$ is hyperbolic if and only if $f$ is pseudo-Anosov as stated in \Cref{t:hyperbolization}. Further we know by the Nielsen--Thurston Classification that $f$ is pseudo-Anosov if and only if $f$ is \emph{infinite-order} and \emph{irreducible}, that is, for any simple closed curve $\gamma$ we have that $f^n(\gamma) \neq \gamma$ for all $n \neq 0$.

Thus, a natural first step to answering \Cref{q:irreducibility} is to give a notion of irreducibility for end-periodic homeomorphisms. This definition was introduced by Handel--Miller in unpublished work and can also be found in \cite{FenleyThesis1989}.

\begin{definition}[Handel--Miller]
    An end-periodic homeomorphism $f: S \to S$ is \emph{irreducible} if it has no:
    \begin{enumerate}
        \item[(1)] \emph{reducing curves}, i.e. a simple closed curve $\gamma$ such that there exists $n, m$ with $n < m$ such that $f^n(\gamma)$ is in a nesting neighborhood of a repelling end and $f^m(\gamma)$ is in a nesting neighborhood of an attracting end (see \Cref{fig:ReducingCurve});
        
        \begin{figure}[H]
            \centering
            \includegraphics[width=3in]{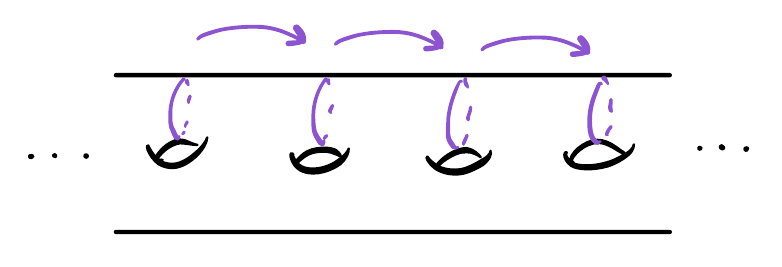}
            \caption{A reducing curve which is sent from a neighborhood of a repelling end into a neighborhood of an attracting end.}
            \label{fig:ReducingCurve}
        \end{figure}
        
        \item[(2)] \emph{AR-periodic lines}, i.e. a line $\ell$ with one end in a nesting neighborhood of an attracting end and the other end in a nesting neighborhood of a repelling end such that there exists $k \neq 0$ with $f^k(\ell) = \ell$ (see \Cref{fig:ReducingArc}); and
        
        \begin{figure}[H]
            \centering
            \includegraphics[width=2in]{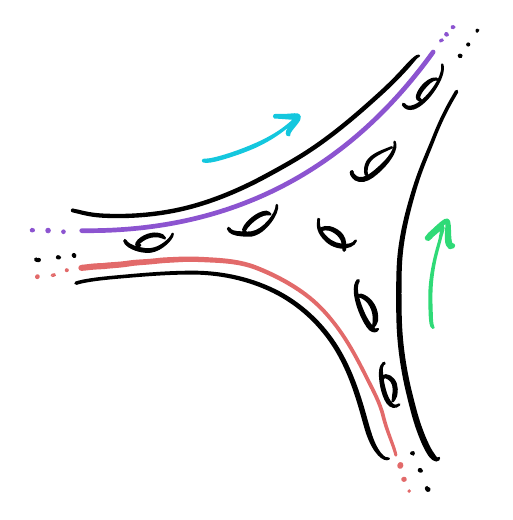}
            \caption{An AR-periodic line is shown in purple and a periodic line is shown in pink. The green and blue arrows indicate the direction the end-periodic homeomorphism is shifting in the ends.}
            \label{fig:ReducingArc}
        \end{figure}
        
        \item[(3)] \emph{periodic curves}, i.e. a simple closed curve $\gamma$ such that $f^n(\gamma) = \gamma$ for some $n \neq 0$ (see \Cref{fig:PeriodicCurve}).
        
        \begin{figure}[H]
            \centering
            \includegraphics[width=2.5in]{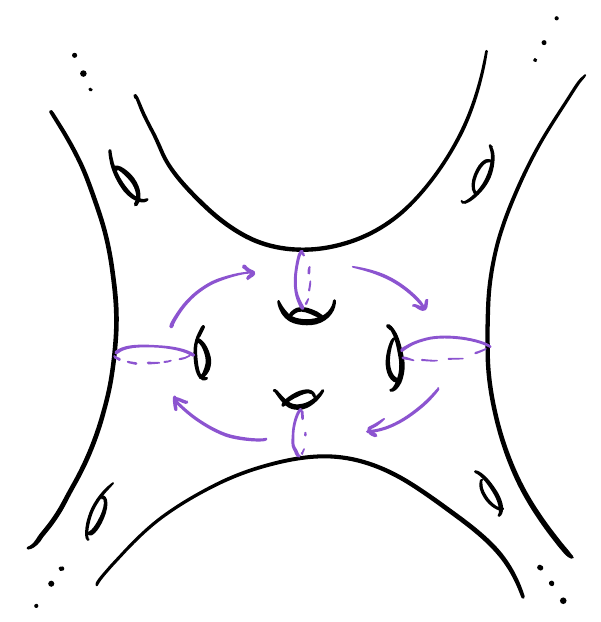}
            \caption{A periodic curve with order $4$.}
            \label{fig:PeriodicCurve}
        \end{figure}
        
    \end{enumerate}
    In \cite{EndPeriodic1}, we introduced the notion of \emph{strong} irreducibility, which disallows any periodic lines including those with both ends in attracting (resp. repelling) ends rather than only AR-periodic lines.
\end{definition}

\subsection{A plethora of examples} \label{subsec:examples}

We are able to give an infinite family of examples of both irreducible and strongly irreducible end-periodic homeomorphisms using a construction detailed in \cite[Section~6.1]{EndPeriodic1}. We say that a subsurface $C \subset S$ is \emph{separating} if there is no component of $S - C$ containing both an attracting and repelling end of $h$. It is \emph{fully separating} if every end of $S$ is contained in a unique component of $S - C$. See \Cref{fig:EndPeriodicExamples} for examples of separating versus fully separating subsurfaces.

\begin{proposition}[\cite{EndPeriodic1}]
    Let $h$ be a composition of disjoint handle shifts on $S$ and let $\rho$ be a partial pseudo-Anosov on a (fully) separating compact subsurface $C \subset S$. If $\rho$ acts with sufficient translation length on the curve complex of $C$ and $C$ has large enough complexity with respect to the amount that $h$ is shifting, then $f = h \rho$ is a (strongly) irreducible end-periodic homeomorphism.
\end{proposition}

An illustration of this construction can be found in \Cref{fig:RevisitExamples}. Note that \emph{strong} irreducibility of $f = h \rho$ corresponds to choosing a \emph{fully} separating compact subsurface $C$.

\begin{figure}
    \centering
    \includegraphics[width=.75\textwidth]{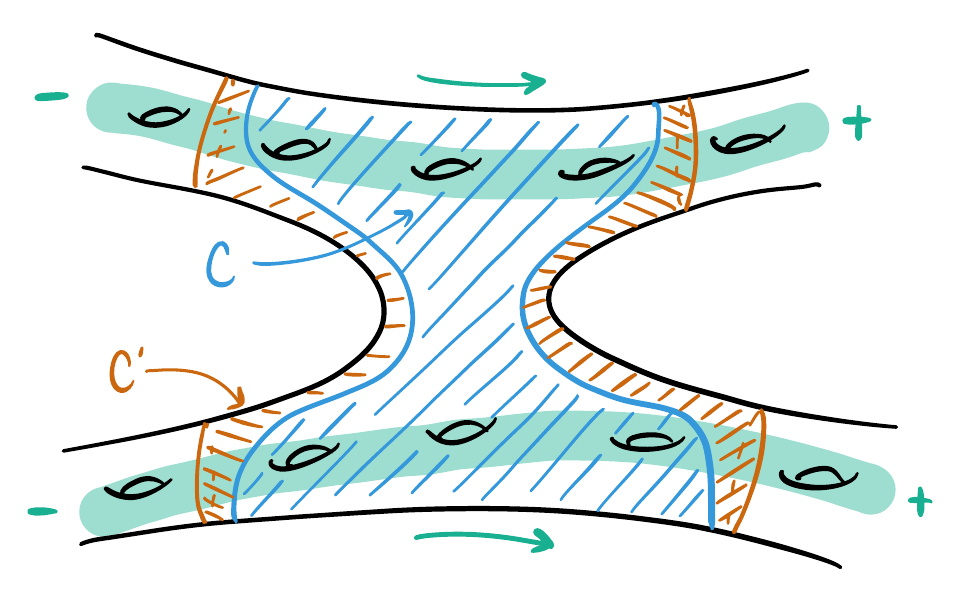}
    \caption{The subsurface $C$ (shown in blue) is separating and the subsurface $C'$ (shown in brown) is fully separating with respect to the handle shifts shown in green.}
    \label{fig:EndPeriodicExamples}
\end{figure}

\subsection{Structure of $\overline M_f$} \label{subsec:CompactStructure}

Our goal for this section is to use irreducibility (resp. strong irreducibility) of $f$ to prove $\overline M_f$ is atoroidal (resp. acylindrical). But first we must address an even more basic question, what actually is $\overline M_f$? We know it exists, but we have yet to discuss its structure. As the notation implies, $\overline M_f$ is a compactification of $M_f$. This is illustrated in \Cref{fig:Compactification} where the compactifying surfaces are denoted by $S_+$ and $S_-$. Where do these compactifying surfaces come from?

\begin{figure}
    \centering
    \includegraphics[width = .8\textwidth]{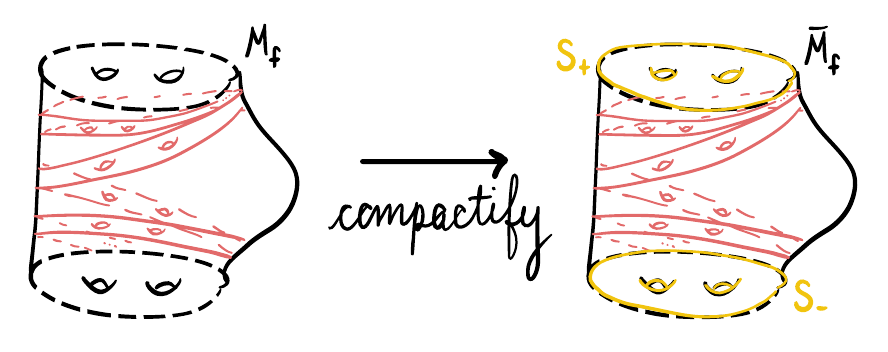}
    \caption{$\overline M_f$ shown as a compactification of $M_f$ with boundary surfaces $S_+$ and $S_-$.}
    \label{fig:Compactification}
\end{figure}

Let $U_-$ be the union of nesting neighborhoods of the repelling ends and let $\mathcal U_- = \bigcup_{n \geq 0} f^n(U_-)$. We call $\mathcal U_-$ the \emph{negative escaping set} for $f$. In the same way define the \emph{positive escaping set} for $f$ to be $\mathcal U_+ = \bigcup_{n \geq 0} f^{-n}(U_+)$ where $U_+$ is the union of nesting neighborhoods of the attracting ends. The following lemma appears (and is proved) in \cite[Lemma~2.4]{EndPeriodic1}.

\begin{lemma} \label{lem:boundary}
    $\mathcal U_+$ and $\mathcal U_-$ are $f$-invariant and $\left<f\right>$ acts freely, properly discontinuously, and cocompactly on each. Thus, the quotients $S_{\pm} = \mathcal U_{\pm} / \left<f\right>$ are closed, orientable surfaces. Furthermore, $\xi(S_+) = \xi(S_-),$ where $\xi(\Sigma_{g,n}) = 3g- 3 + n$ is the complexity. 
\end{lemma}

Now to define $\overline M_f$. First we will identify the infinite cyclic cover, $M_{\infty}$, of $M_f$ dual to the fibration over $S_1$ with $p: S \times (-\infty, \infty) \to M_f$ together with the covering action $F(x,t) = (f(x), t-1)$. This has quotient $M_f$. 

We then define a partial compactification of $S \times (-\infty, \infty)$, denoted $\widetilde M_{\infty}$ inside of $S \times [-\infty, \infty]$ as follows \[\widetilde M_{\infty} = S \times (-\infty, \infty) \sqcup (\mathcal U_+ \times {\infty}) \sqcup (\mathcal U_- \times {-\infty}).\] By \Cref{lem:boundary}, $\mathcal U_+$ and $\mathcal U_-$ are $f$-invariant so $F(x, \pm \infty) = (f(x), \pm \infty)$ when $x \in \mathcal U_{\pm}$. Finally, we let $\overline M_f = \widetilde M_{\infty}/ \left< F\right>$ with $\partial \overline M_f = S_+ \sqcup S_-$.

We are now prepared to address our guiding goal for this section of using irreducibility (resp. strong irreducibility) of $f$ to prove $\overline M_f$ is atoroidal (resp. acylindrical).

\begin{proposition}[\cite{EndPeriodic1}]
    If $f: S \to S$ is an irreducible end-periodic homeomorphism, then $\overline M_f$ is atoroidal. 
\end{proposition}

\begin{proof}[Sketch of Proof]
    Suppose $\overline M_f$ is not atoroidal. Then there exists an incompressible torus $T$ which is not boundary parallel. By \Cref{thm:RoussarieThurston}, we can arrange $T$ to be transverse to the fiber $S \times \{0\}$. Cutting $\overline M_f$ along $S \times \{0\}$ we see that $T$ gives rise to a periodic curve which contradicts irreducibility. See \Cref{fig:Atoroidal}.
\end{proof}

\begin{figure}[ht]
    \centering
    \includegraphics[width = .5\textwidth]{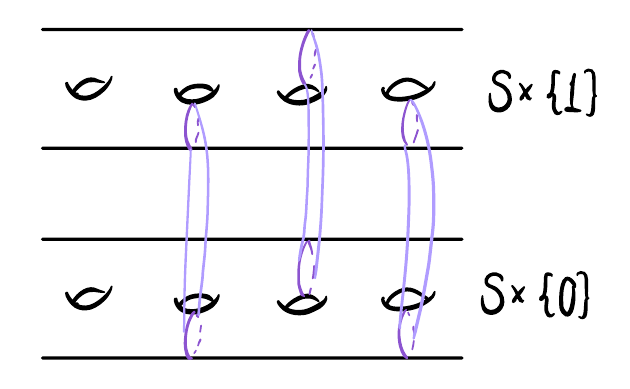}
    \caption{An incompressible torus giving rise to a periodic curve.}
    \label{fig:Atoroidal}
\end{figure}

Note that irreducibility is in fact a stronger hypothesis than is needed to ensure $\overline M_f$ is atoroidal since only periodic curves will give rise to incompressible tori in $\overline M_f$. As a consequence of this, end-periodic homeomorphisms without periodic curves are often called \emph{atoroidal}. See, for example, \cite{LandryMinskyTaylor2023, Whitfield}.

\begin{proposition}[\cite{EndPeriodic1}]
    If $f:S \to S$ is a strongly irreducible end-periodic homeomorphism, then $\overline M_f$ is acylindrical.
\end{proposition}

The argument for aclyindricity is more subtle than the one we used to show that $\overline M_f$ is atoroidal since there are two ways for cylinders to arise: 1) reducing curves, and 2) periodic lines (this is where strong irreducibility is needed!). The vague idea is that you once again apply \Cref{thm:RoussarieThurston} to put cylinders into their normal form and then lift to $\widetilde M_{\infty}$ to analyze these cylinders more carefully. You can see some of the different potential configurations of cylinders in \Cref{fig:Acylindrical}.

\begin{figure}[ht]
    \centering
    \includegraphics[width = .8\textwidth]{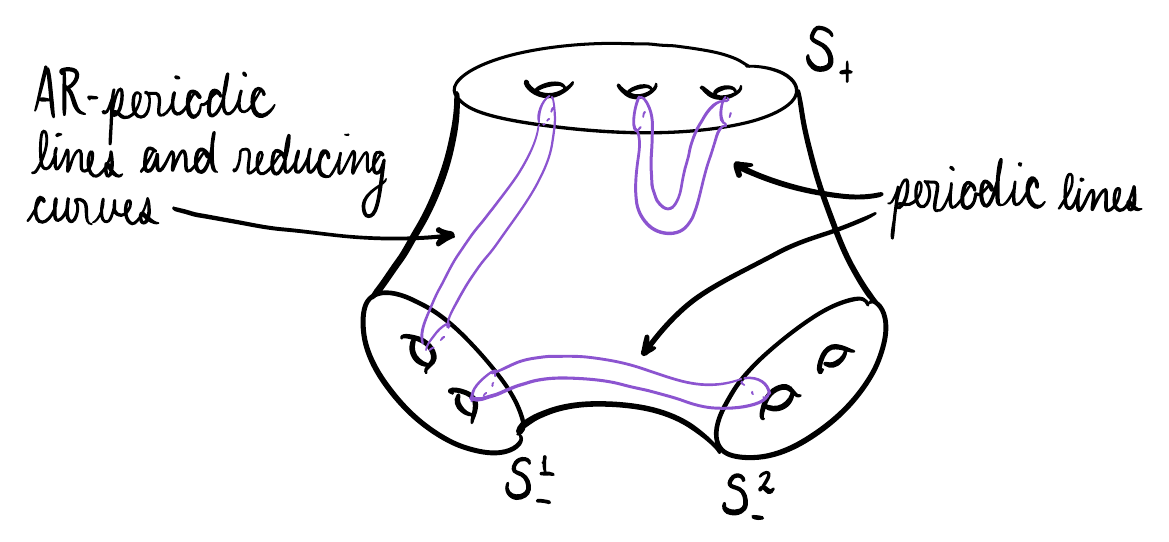}
    \caption{The cylinder on the left is labeled ``AR-periodic lines and reducing curves" because it could arise from either phenomenon since it runs from a component of $\partial \overline M_f$ corresponding to a repelling end of $f: S \to S$ to a component of $\partial \overline M_f$ corresponding to an attracting end of $f: S \to S$. Whether a cylinder arises from an AR-periodic line or a reducing curve can be seen after lifting to $\widetilde M_{\infty}$. The cylinders on the right are labelled as ``periodic lines" since they correspond to periodic lines running between end of the same type. }
    \label{fig:Acylindrical}
\end{figure}

\section{A lower bound on volume} \label{sec:lower-bound}

Given a taut \emph{depth-one} foliation we are now prepared to address \Cref{q:taut-geometry}! In particular, as we saw in the previous section, we can say a lot about the geometry of $\overline M_f$. It turns out we can do more than simply determine when $\overline M_f$ admits a hyperbolic metric, we can actually give bounds on the volume of that hyperbolic metric in terms of the action of $f$ on the \emph{pants graph}. The \emph{pants graph} $\mathcal P(S)$ of a surface $S$ is a graph with vertices corresponding to pants decompositions of $S$ and edges corresponding to \emph{elementary moves}. These elementary moves come in two flavors as illustrated in \Cref{fig:ElementaryMoves}. We will call an elementary move that takes place in a one-holed torus a \emph{T-move} and an elementary move that takes place in a four-holed sphere an \emph{S-move}. We equip the pants graph with a slightly different path metric than the usual one. In particular, we endow edges corresponding to $S$-moves with length $2$ and those corresponding to $T$-moves with length $1$.

\begin{figure}
    \centering
    \includegraphics[width = 0.5\textwidth]{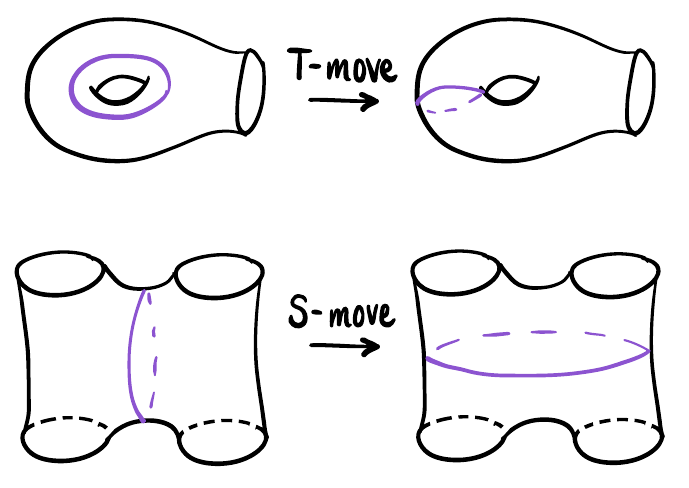}
    \caption{Two vertices in the pants graph are connected by an edge if their associated pants decompositions are related by one of the two elementary moves shown here.}
    \label{fig:ElementaryMoves}
\end{figure}

The upper bound in the below theorem was proved by Field--Kim--Leininger--Loving \cite{EndPeriodic1} and the lower bound was proved by Field--Kent--Leininger--Loving \cite{EndPeriodic2}. 

\begin{theorem} \label{thm:VolumeBounds}
    Suppose $f: S \to S$ is a strongly irreducible end-periodic homeomorphism. Then \[c_2 \cdot \tau(f) \leq \Vol (\overline M_f) \leq c_1 \cdot \tau(f)\] where $c_1$ is the volume of an ideal octahedron, $c_2$ is a constant depending only on the \emph{capacity} of $f$, and $\tau(f)$ is the \emph{asymptotic translation length} of $f$ on the pants graph $\mathcal P(S)$.
\end{theorem}

We will define and discuss the \emph{capacity} of $f$ at some length, but for now it is enough to say that it captures the same type of information as the notion of topological complexity does for a finite-type surface. The \emph{asympototic translation length} of $f$ on $\mathcal P(S)$ is defined as \[\tau(f) = \inf_{P \in \mathcal P(S)} \liminf_{n \to \infty} \frac{d(P, f^n(P))}{n}.\] It is straightforward to show that $\tau(f^n) = n \cdot \tau(f)$ for all $n \geq 1$. The astute reader might note that $\mathcal P(S)$ is not connected when $S$ is infinite type. See \Cref{fig:PantsDecomp} for an example of two pants decompositions which are in different connected components of $\mathcal P(S)$. Thus, we really are taking this infimum over the union of connected components for which $d(P, f^n(P))$ is finite for some $n > 0$. Note that \Cref{thm:VolumeBounds} is an infinite-type analogue of the following theorem of Brock.

\begin{figure}
    \centering
    \includegraphics[width = \textwidth]{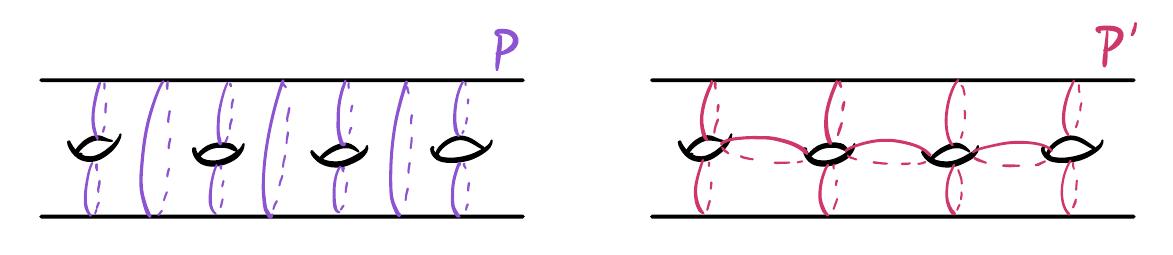}
    \caption{Let $S$ be the ladder surface. Then $P$ and $P'$ are in different connected components of $\mathcal P(S)$.}
    \label{fig:PantsDecomp}
\end{figure}

\begin{theorem}[\cite{Brock-mappingtorus-vol}] \label{thm:Brock}
    Let $f: \Sigma \to \Sigma$ be a pseudo-Anosov homeomorphism. There exists constants $K_1, K_2$ depending only on the topology of $\Sigma$ such that \[K_1 \cdot \tau(f) \leq \Vol(M_f) \leq K_2 \cdot \tau(f).\]
\end{theorem}

Before beginning to discuss the proof of the lower bound in \Cref{thm:VolumeBounds} it will be illuminating to briefly sketch Brock's proof of the lower bound in \Cref{thm:Brock}. My goal is to first highlight the big ideas of his proof, then note where some key complications arise when attempting to adapt it the infinite-type setting and, finally, to introduce the notion of \emph{well-pleated surfaces}.

\subsection{Brock's lower bound} Let $f: \Sigma \to \Sigma$ be a pseudo-Anosov homeomorphism and $M_f$ its corresponding hyperbolic mapping torus.

\medskip

\noindent \underline{Big Ideas of Proof:}

\begin{enumerate}
    \item[1)] The driving idea is that bounded length curves in $M_f$ contribute a definite amount to the hyperbolic volume of $M_f$. This is illustrated in \Cref{fig:MargulisTube}.
    \item[2)] Thus, a lower bound on the number of bounded length curves will give a lower bound on the volume of $M_f$.
    \item[3)] Hence, obtaining a lower bound on the number of bounded length curves in terms of a path in $\mathcal P(\Sigma)$ will give us the desired result.
\end{enumerate}     

\begin{figure}
    \centering
    \includegraphics[width = \textwidth]{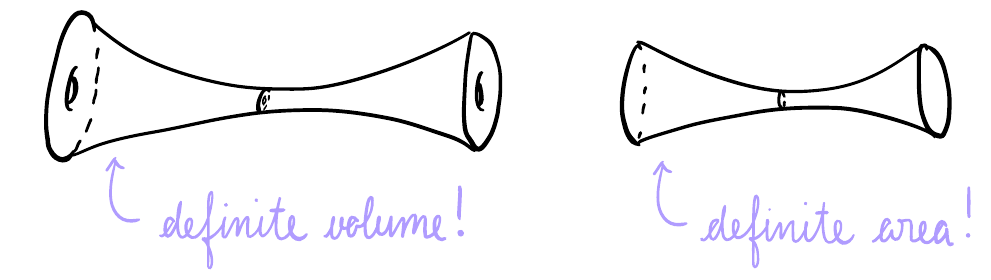}
    \caption{Recall that in a surface the collar lemma tells us that a short simple closed curve lives in a collar with definite area. The analogue of this in $3$-dimensions is the Margulis lemma which tells us that a short curve in a $3$-manifold is contained in a solid torus, called a Margulis tube, with definite volume.}
    \label{fig:MargulisTube}
\end{figure}

Of course this sketch of proof begs the question, how does Brock find these bounded length curves in $M_f$ to begin with? The short answer is that he uses an interpolation through \emph{simplicial hyperbolic surfaces} \cite{Bonahon, Canary.CoveringTheorem}. Roughly speaking, a simplicial hyperbolic surface is a path-isometric mapping from a singular hyperbolic surface to a hyperbolic $3$-manifold that is totally geodesic in the complement of a triangulation, and an ``interpolation" is a one-parameter family of such maps. 

This interpolation takes place in the infinite cyclic cover, $\Sigma \times \mathbb R$, of $M_f$ as shown in \Cref{fig:BrockPleatedSurfaces}. Why do we pass to $\Sigma \times \mathbb R$? Brock understands this picture very well since it puts us in the setting of convex cores of quasi-Fuchsian hyperbolic $3$-manifolds \cite{Brock-convex-vol}. If we start with a bounded length pants decomposition on the initial surface, then this interpolation produces a whole sequence of bounded length pants decompositions. Next, Brock shows that the number of bounded length curves coming from this sequence gives an upper bound on distance in the pants graph between the bounded length pants decomposition on the initial surface in the interpolation and the pants decomposition on the final surface, and hence on $\tau(f)$.

\begin{figure}
    \centering
    \includegraphics[width = .9\textwidth]{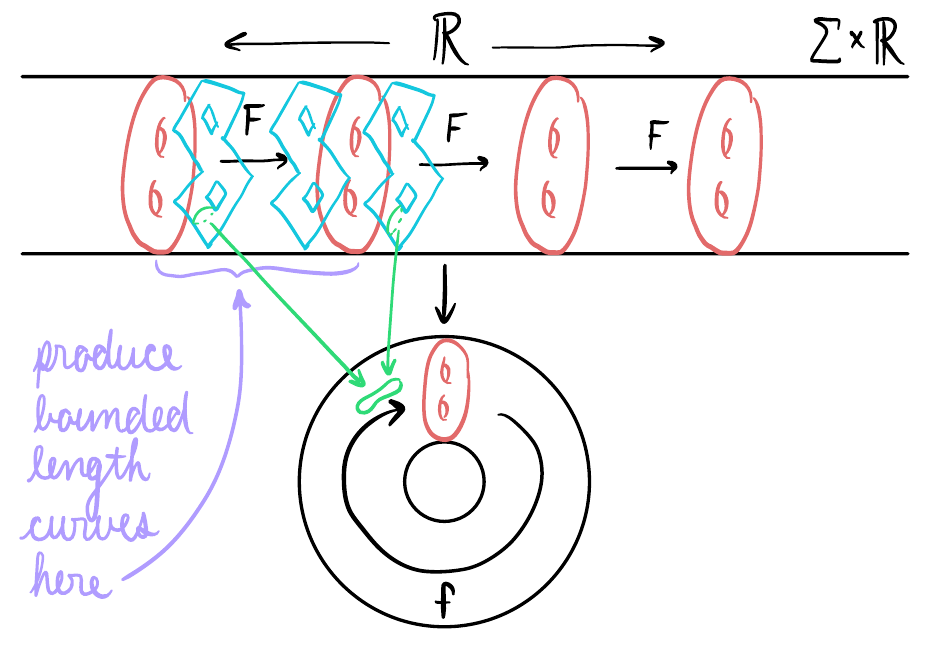}
    \caption{The interpolation through simplicial hyperbolic surfaces is shown by the crumpled blue surfaces in $\Sigma \times \mathbb R$ where the covering transformation is given by $F(x, t) = (f(x), t+1)$.}
    \label{fig:BrockPleatedSurfaces}
\end{figure}

\medskip

\noindent \underline{Difficulty:} The interpolation overlaps with its image under the covering transformation $F$, which means that many curves in this sequence could project to the same curve in $M_f$. This is illustrated by the green curves in \Cref{fig:BrockPleatedSurfaces}.

\medskip

\noindent \underline{Solution:} Pass to powers and apply a limiting argument!

\subsection{Apply Brock's argument to the infinite-type setting}

Our goal in this section will be to replicate Brock's argument in the infinite-type setting! The primary driving idea that bounded length curves contribute definite volume still applies. However, passing to powers in order to get a lower bound on volume in terms of the number of bounded length curves presents some issues in the infinite-type setting. In addition, the lower bound on the number of bounded length curves in terms of a path in $\mathcal P(\Sigma)$ given by Brock depends on $\chi(\Sigma)$ which no longer makes sense for an infinite-type surface. 

We will first address this second issue by developing a notion of complexity for $f: S \to S$ an end-periodic homeomorphism. Let's return to our family of examples from \Cref{subsec:examples}, which we recall in \Cref{fig:RevisitExamples}.

\begin{figure}
    \centering
    \includegraphics[width = \textwidth]{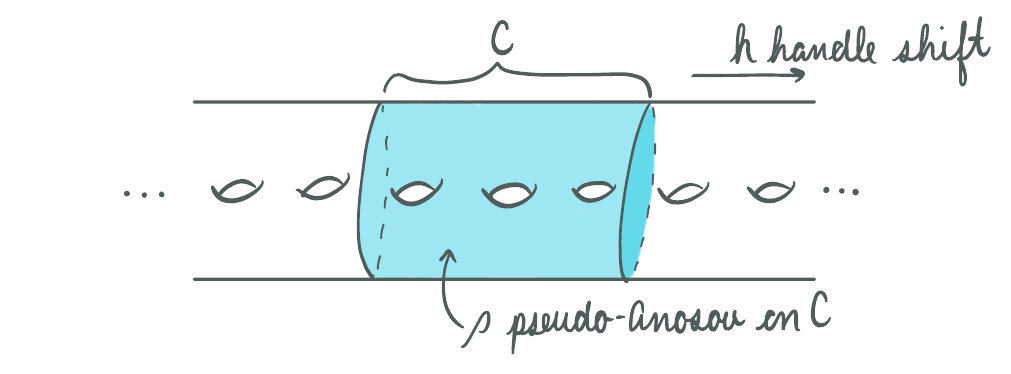}
    \caption{The homeomorphism $f = h \rho$ is strongly irreducible so long as $\rho$ acts with sufficiently large translation length on the curve complex of the subsurface $C$.}
    \label{fig:RevisitExamples}
\end{figure}

There are two immediate pieces of information that are key to the structure of these examples:

\begin{enumerate}
    \item[1)] how much is $h$ shifting by?
    \item[2)] how big is $C$?
\end{enumerate}

It turns out that the amount that $h$ is shifting by is easily gleaned by the complexity of the boundary of $\overline M_f$, which we denote by $\xi(f)$. Note that \[\xi(f) = \xi(\partial \overline M_f) = \xi(S_+) + \xi(S_-) = 2 \xi(S_+).\] We call $\xi(f)$ the \emph{end complexity} of $f$. 

On the other hand, ``how big is $C$?" is not quite the right thing to measure! Instead we need to ask the following,

\begin{enumerate}
    \item[2')] what is the size of the smallest subsurface of $S$ where stuff gets ``trapped", i.e. does not escape into either the attracting or repelling ends under forward or backward iteration of $f$? This is illustrated in \Cref{fig:FlatCoveringPicture}. 
\end{enumerate}

\begin{figure}
    \centering
    \includegraphics[width = .8\textwidth]{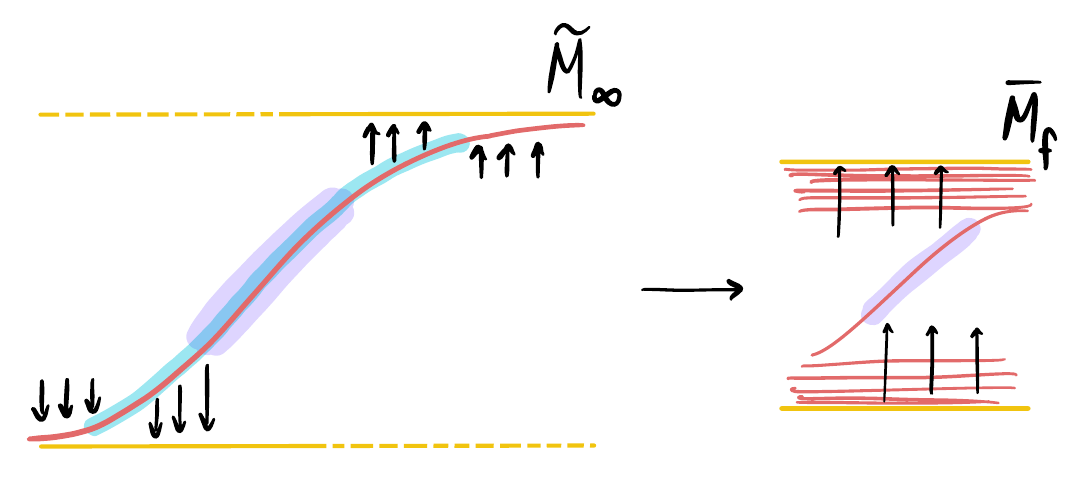}
    \caption{An illustration (drawn down a dimension) of the parts of a fiber $S$ in $\overline M_f$ that can't be pushed into $\partial \overline M_f$.}
    \label{fig:FlatCoveringPicture}
\end{figure}

But how do we quantify this? First recall the notation we established in \Cref{subsec:CompactStructure} where $U_+$ (resp. $U_-$) is the union of nesting neighborhoods of the attracting ends (resp. repelling ends) of $S$ with respect to $f$. We will call $Y = S - (U_- \cup U_+)$ a \emph{core} for $f$ and define the \emph{core characteristic} of $f$ to be the size of a minimal core for $f$, i.e. \[\chi(f) = \max_{Y \subset S} \chi(Y),\] where the maximum is taken over all cores $Y \subset S$ for $f$. Note that we are taking a maximum to find the \emph{minimal} core simply because the Euler characteristic of a core is always negative.

We can now define our replacement for the topological complexity of a finite-type surface, which is the tuple $(\chi(f), \xi(f))$ consisting of the core characteristic together with the end complexity of $f$. We call this pair the \emph{capacity} of $f$. 

\subsection{Well-pleated surfaces}

An important tool in our proof (which will ultimately allow us to deal with the issue of passing to powers) is the notion of \emph{well-pleated surfaces}. 

A \emph{pleated surface} is very similar to a simplicial hyperbolic surface with the triangulation of the singular hyperbolic surface replaced with a geodesic lamination \cite{epstein-marden.2006}. See \Cref{fig:Spinning} for an illustration of how you can pass from a triangulation to a geodesic lamination using a technique called \emph{spinning}. Technically, a \emph{pleated surface} is a map $\varphi: \Sigma \to N$ from a finite-type surface $\Sigma$ to a hyperbolic $3$-manifold $N$ together with a choice of hyperbolic metric, $\sigma$, on $\Sigma$ and a geodesic lamination $\lambda$ on $(\Sigma, \sigma)$ so that:

\begin{enumerate}
    \item[(1)] $\varphi$ is length preserving on paths,
    \item[(2)] $\varphi$ maps leaves of $\lambda$ to geodesics, and 
    \item[(3)] $\varphi(\Sigma - \lambda)$ is a totally geodesic submanifold of $N$.
\end{enumerate}

If $\lambda$ is the smallest geodesic lamination satisfying (2) and (3), we call it the \emph{pleating locus} for $\varphi$.

\begin{figure}
    \centering
    \includegraphics[width = \textwidth]{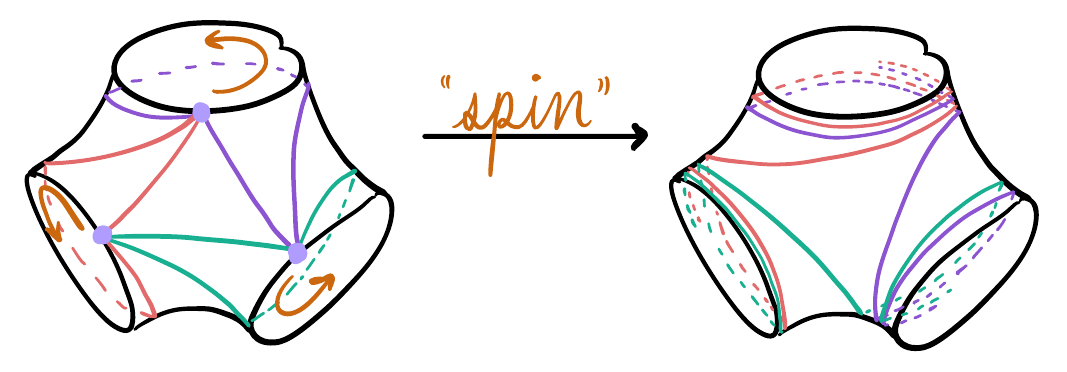}
    \caption{Note that you can go from a triangulation or arc system to a lamination by ``spinning".}
    \label{fig:Spinning}
\end{figure}


In \cite{EndPeriodic2}, we introduced the notion of a well-pleated surface as a natural adaptation of a pleated surface to the infinite-type setting. Let $f:S \to S$ be an end-periodic homeomorphism. A \emph{well-pleated surface adapted to the core $Y$} in $\overline M_f$ is a pleated surface $\varphi: (S, \sigma) \to \overline M_f$, homotopic to the inclusion of $S$ into $\overline M_f$, such that

\begin{enumerate}
    \item[1)] the pleating locus is contained in a \emph{pants-lamination}, and  
    \item[2)] $\varphi(S - Y) \subset \partial \overline M_f$.
\end{enumerate}

Here a \emph{pants-lamination} of a hyperbolic surface $X$ is a lamination whose leaves consist of a pants decomposition of $X$ together with the spiralling leaves in each pair of pants that are shown on the right of \Cref{fig:Spinning}. Additionally, although the ``inclusion" of $S$ into $\overline M_f$ is only well-defined up to precomposing with powers of $f$, this ambiguity does not exist in $\widetilde M_{\infty}$ where we will actually work. 

The goal is to build an interpolation through pleated surfaces to produce bounded length curves in $\overline M_f$ using bounded length pants decompositions of our well-pleated surface. We will let $\Omega(f, Y)$ be the set of all well-pleated surfaces adapted to a fixed core $Y$.

\begin{lemma}[{\cite[Lemma~4.1]{EndPeriodic2}}]
    Suppose that $Y$ is a minimal core of the strongly irreducible, end-periodic homeomorphism $f: S \to S$ and suppose $(\varphi:(S, \sigma) \to \overline M_f) \in \Omega(f, Y)$, then the total $\sigma$-length of the boundary of $Y$ satisfies $\ell_{\sigma}(\partial Y) \leq \frac{2\pi |\chi(Y)|}{\beta}$.
\end{lemma}

Here $\beta = \beta(\xi(f)) > 0$ is a constant which, by work of Basmajian \cite[Theorem~1.1]{Basmajian}, guarantees that the $\beta$-neighborhood of $\partial \overline M_f$ is a product, $\partial \overline M_f \times [0, \beta]$. Although we have omitted most proofs in this paper since it is intended to be expository in nature and focused on intuition rather than details, we will provide a sketch of the proof of this lemma. It is quite brief and emphasizes the importance of considering \emph{minimal} cores. A cartoon of a minimal core is shown by the purple highlighted portion of the fiber $S$ in \Cref{fig:FlatCoveringPicture}.

\begin{proof}[Sketch of Proof]
    Let $\varepsilon > 0$ be maximal so that the $\varepsilon$-neighborhood of $\partial Y$ are annuli. Note that $\Area(Y, \sigma) = 2\pi |\chi(Y)|$, by Gauss--Bonnet, and $\Area(N_{\varepsilon}(\partial Y)) = \sinh(\varepsilon) \ell_{\sigma}(\partial Y) > \varepsilon \ell_{\sigma}(\partial Y)$, by some basic hyperbolic trigonometry \cite{Buser-book}. Since $\Area(N_{\varepsilon}(\partial Y)) \leq \Area(Y, \sigma)$, then $\varepsilon \ell_{\sigma}(\partial Y) \leq 2\pi|\chi(Y)|$.

    Thus, if $\ell_{\sigma}(\partial Y) > \frac{2\pi |\chi(Y)|}{\beta}$, then $\varepsilon < \beta$. This implies that the $N_{\varepsilon}(\partial Y)$ is properly homotopic into $N_{\beta}(\partial \overline M_f).$ Thus, $Y - N_{\varepsilon}(\partial Y)$ is a core for $f$ with $|\chi(Y - N_{\varepsilon}(\partial Y))| < |\chi(Y)|$, which contradicts our assumption that $Y$ was minimal. 
\end{proof}


If $Y$ is a core for $f$ with $\chi(Y) = \chi(f)$, then any $\varphi \in \Omega(f, Y)$ for which $|\chi(Y)|$ is minimized is called \emph{minimally well-pleated} and we can conclude that $\ell_{\sigma}(\partial Y) < \frac{2\pi |\chi(Y)|}{\beta}$.

\begin{figure}
    \centering
    \includegraphics[width = .8\textwidth]{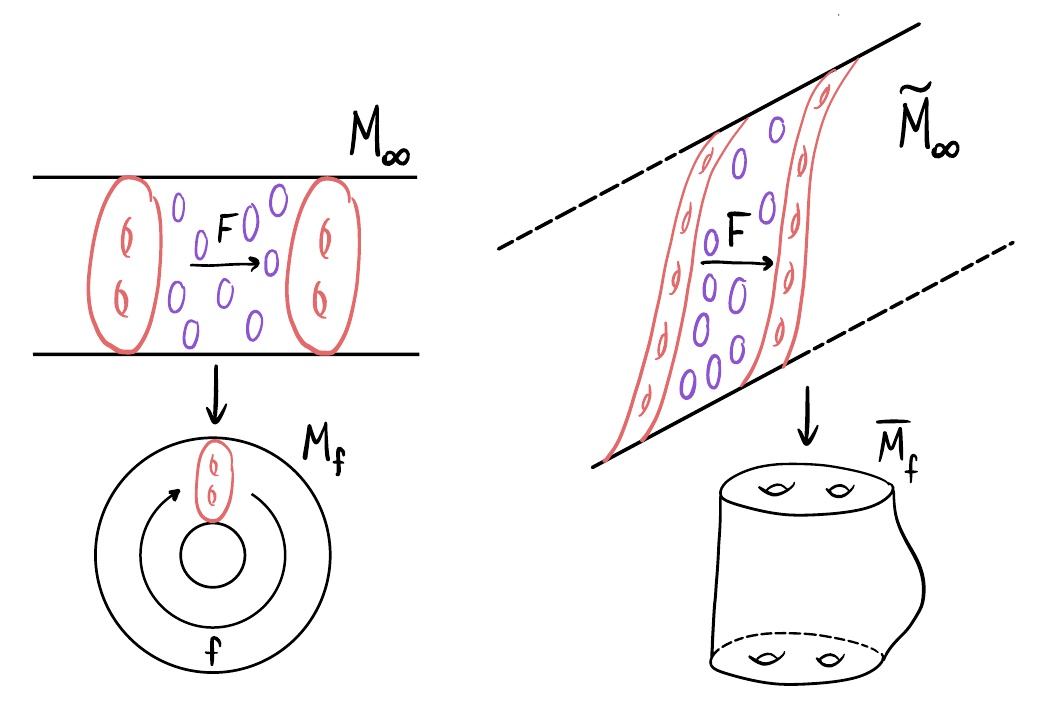}
    \caption{The covering transformation in both the finite-type and infinite-type setting is given by $F(x,t) = (f(x), t-1)$.}
    \label{fig:TwoInterpolations}
\end{figure}

So what does an interpolation through minimally well-pleated surfaces actually look like in $\widetilde M_{\infty}$? This is illustrated in \Cref{fig:BoundingBoundaries}. As you can see, the subsurface of $S$ on which we need to find a bounded length pants decomposition continues to grow as we move further into our interpolation. This is very different from what happens in the finite-type setting (see \Cref{fig:BrockPleatedSurfaces}). A cartoon illustration of this difference between the finite and infinite-type setting is also shown in \Cref{fig:TwoInterpolations} which reproduces \cite[Figure~1]{EndPeriodic2}.

\begin{figure}
    \centering
    \includegraphics[width = .9\textwidth]{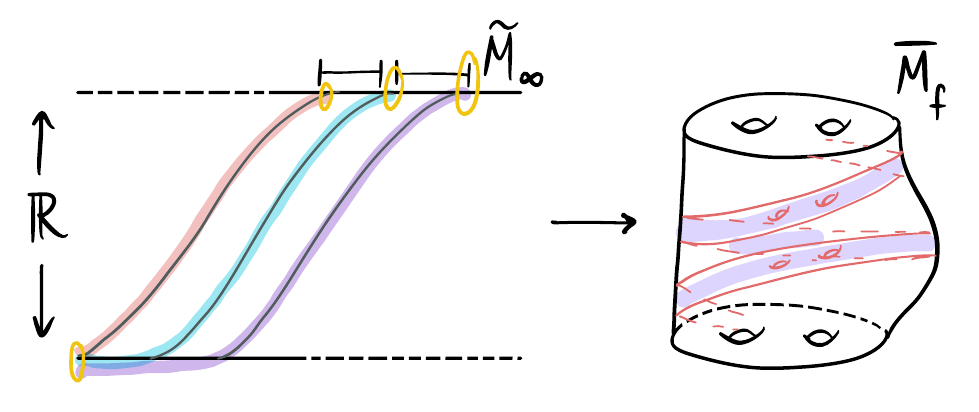}
    \caption{On the left is an illustration of the interpolation through pleated surfaces in $\widetilde M_{\infty}$. Note that as we move through this interpolation we are ``unzipping" along the top boundary of $\widetilde M_{\infty}$ and ``zipping up" along the bottom boundary of $\widetilde M_{\infty}$. Hence, the subsurface of S where we need to find a bounded length pants decomposition keeps growing (this is indicated by the yellow circles) as we move through our interpolation.}
    \label{fig:BoundingBoundaries}
\end{figure}

Fortunately, we can pass to a uniform power $K \geq 5 |\chi(f)|$ such that $Y$ is a core for $f^K$. We can then enlarge $Y$ to include a fundamental domain for the action of $f^K$ on $U_-$. This enlargement will support our interpolation and we will still have a bound on the length of its boundary. 

Why is this true? In short, this is due to the irreducibility of $f$! Since $f$ is irreducible, there is a uniform power $K$ (the same $K$ as above) such that if part of a curve $\alpha \subset Y$ leaves $Y$ for some forward iterate of $f$ (i.e. $f^{k_0}(\alpha) \cap U_+ \neq \emptyset$ for some $k_0 > 0$), then it does so for all sufficiently large forward iterates (i.e. $f^k(\alpha) \cap U_+ \neq \emptyset$ for $k \geq K$). This is proved in \cite[Lemma~3.1]{EndPeriodic2}.

Because $Y$ has bounded topological type it admits a bounded length pants decomposition. This is a relative version of Ber's Lemma \cite{BersP1, BersP2} which was proved by Parlier \cite[Theorem~1.1]{Parlier2023}.

In addition, none of the higher powers of these bounded length pants curves are homotopic into $Y$. Thus, no two bounded length curves we care about in $\widetilde M_{\infty}$ project to the same curve in $\overline M_{f^K}$. This avoids the issue illustrated by the green curves in \Cref{fig:BrockPleatedSurfaces} which Brock must address via a limiting argument that is not available in our setting. 

\subsection{Putting the pieces together}

Our proof of the lower bound now proceeds by building an interpolation through simplicial hyperbolic surfaces, guided by our minimally well-pleated surfaces, with the volume argument proceeding as in Brock, but without the need to take a limit. This is sketched in slightly more detail below.

\begin{enumerate}
    \item[1)] Use minimally well-pleated surfaces to find bounded length pants decompositions $P_{\alpha}$ and $P_{\omega}$ on $\phi_{\alpha}:(\Sigma, \sigma_{\alpha}) \to \widetilde M_{\infty}$ and $\phi_{\omega}:(\Sigma, \sigma_{\omega}) \to \widetilde M_{\infty}$, respectively.
    \item[2)] Build an interpolation through simplicial hyperbolic surfaces that interpolate from between these minimally well-pleated surfaces equipped with the bounded length pants decompositions $P_{\alpha}$ and $P_{\omega}$. 
    \item[3)] This interpolation gives us a path $P_{\alpha} = P_1, P_2, \ldots, P_n = P_{\omega}$ in the pants graph between $P_{\alpha}$ and $P_{\omega}$ through bounded length curves.
    \item[4)] From Brock \cite{Brock-mappingtorus-vol} we have the following bound: \[\tau(f) \leq d_{\mathcal P}(P, f^{-K}(P)) \leq d_{\mathcal P}(P_{\alpha}, P_{\omega}) \leq K \cdot |P_1 \cup \cdots \cup P_n|,\] where $K$ relies only on the capacity of $f$.
    \item[5)] We are done! Modulo some volume arguments that are needed to make our intuition about short curves controlling volume precise. 
\end{enumerate}


\section{An upper bound on volume} \label{sec:upper-bound}

Our approach to proving the upper bound in \Cref{thm:VolumeBounds} will follow a proof of Brock's upper bound given by Agol \cite{Agol-oct}. In particular, we will build a model for $\overline M_f$ using a block decomposition obtained from a path in $\mathcal P(S)$.

\subsection{The building blocks}

We will build our model for $\overline M_f$ out of \emph{pants blocks}, which come in two flavors corresponding to the two types of elementary moves that give the edge relations in $\mathcal P(S)$. More precisely, these pants blocks will be built as quotients of $\Sigma \times I$ where $\Sigma \cong \Sigma_{1,1}$ or $\Sigma \cong \Sigma_{0,4}$ and $I = [0,1]$. Up to homeomorphism there are two distinct pants blocks, $\mathcal B^T$ and $\mathcal B^S$, as shown in \Cref{fig:PantsBlock}.

\begin{figure}
    \centering
    \includegraphics[width = .9\textwidth]{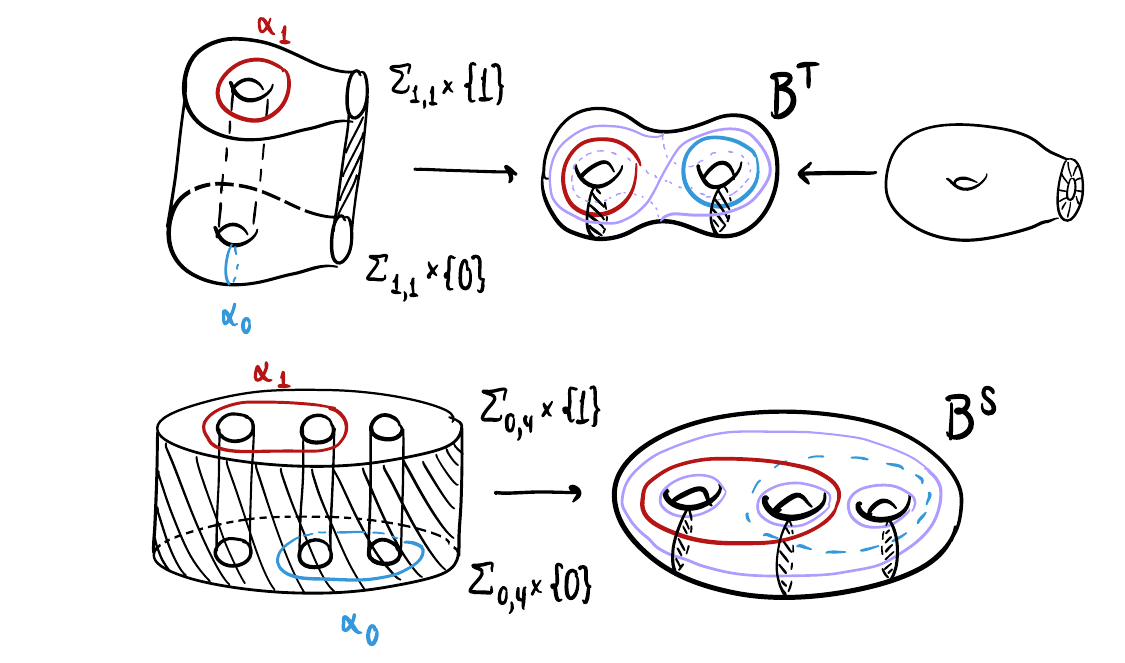}
    \caption{On the top is the pants block $\mathcal B^T$ obtained as a quotient of $\Sigma_{1,1} \times I$ and on the the bottom is the pants block $\mathcal B^S$ obtained as a quotient of $\Sigma_{0,4} \times I$.}
    \label{fig:PantsBlock}
\end{figure}

To be completely precise about the definition of these blocks we would need to introduce the notion of \emph{pared manifold}. This is important for hyperbolization as illustrated by the following theorem \cite{Thurston.Geometrization.1, mcmullen1992riemann, Morita.1984}. However, instead of going into detail, we will just say briefly that we will often consider a pair $(M, \mathfrak p)$ where $\mathfrak p$ is a closed subset of $\partial M$ called the \emph{paring locus} which consists of tori and annuli. This will be helpful notationally in our discussion. 

\begin{theorem}[Thurston]
   Suppose $(M, \mathfrak p)$ is a pared $3$-manifold and $\partial M - \mathfrak p \neq \emptyset$ is incompressible. Then $M - \mathfrak p$ admits a convex hyperbolic metric $\sigma$. If $(M, \mathfrak p)$ is additionally assumed to be acylindrical and non-degenerate, then there is a unique (up to isometry) convex hyperbolic structure, $\sigma_{min}$, on $M - \mathfrak p$ for which $\partial M - \mathfrak p$ is totally geodesic.
\end{theorem}

The blue, red, and purple curves drawn in each of $\mathcal B^S$ and $\mathcal B^T$ in \Cref{fig:PantsBlock} are the paring loci coming from $\alpha_0, \alpha_1$, and $\partial \Sigma \times I$, respectively. Note that there is a natural \emph{stratification} of the pants blocks given by letting $\partial_1 \mathcal B$ be the paring locus of $\mathcal B$ (note that this is just the collection of boundary components of three-holed spheres shown in \Cref{fig:Stratified}) and letting $\partial_2 \mathcal B = \partial \mathcal B - \partial_1 \mathcal B$ (note that this is simply the interior of the three-holed spheres shown in \Cref{fig:Stratified}). Thus, the stratification of $\mathcal B$ is $\partial_1 \mathcal B$, $\partial_2 \mathcal B$, and $\intt(\mathcal B)$. 

We have the following geometric data about our pants blocks; see e.g.~Agol \cite[Lemma 2.3]{Agol-oct}. Note that $V_{oct}$ is the volume of a regular ideal octahedron.

\begin{proposition}[Agol, \cite{Agol-oct}] \label{p:agol-blocks}
    $\mathfrak B^S - \partial_1 \mathfrak B^S$ (respectively $\mathfrak B^T - \partial_1 \mathfrak B^T$) has a convex hyperbolic metric $\sigma_{\mathcal B^S}$ (respectively $\sigma_{\mathcal B^T}$) with totally geodesic boundary consisting of a union of thrice-punctured spheres. Furthermore, $Vol(\mathfrak B^S - \partial_1 \mathfrak B^S, \sigma_{B^T}) = V_{oct}$ and $Vol(\mathfrak B^T - \partial_1 \mathfrak B^T, \sigma_{B^S}) = 2 V_{oct}$.
\end{proposition}

\begin{figure}
    \centering
    \includegraphics[width = .75\textwidth]{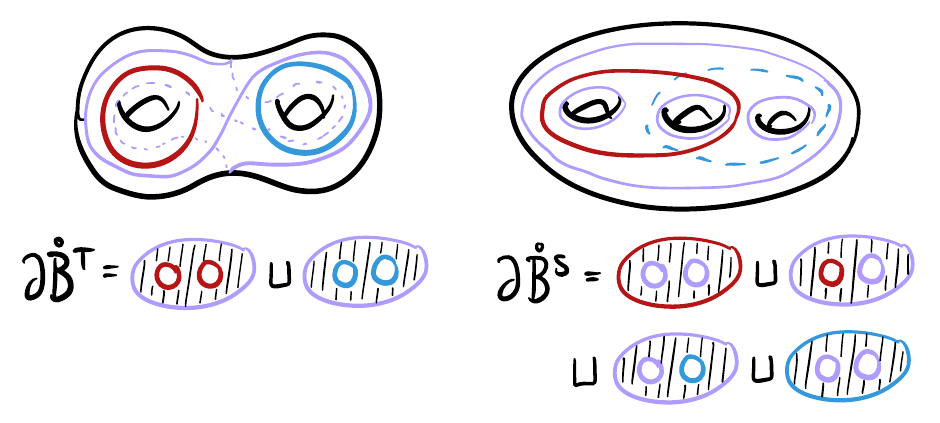}
    \caption{The boundaries of $\mathring {\mathfrak B^S}$ and $\mathring {\mathfrak B^T}$ both consist of a disjoint union of three-holed spheres (i.e. pairs of pants) as shown. }
    \label{fig:Stratified}
\end{figure}

\subsection{Block decompositions}

Suppose $N$ is a compact $3$-manifold, possibly with boundary, and $L \subset N$ is a link. A \emph{block decomposition} of $N$ relative to $L$ is a finite collection of pants blocks $\{ \mathcal B_i\}_{i = 1}^n$ so that: 

\begin{enumerate}
    \item[1)] the interiors of the $\mathcal B_i$ are pairwise disjoint;
    \item[2)] the union of the blocks is all of $N$; and
    \item[3)] whenever two blocks intersect, they do so in a union of components of the strata.
\end{enumerate}

In other words, $N$ is built from finitely many pants blocks glued together in pairs along closures of three-holed spheres in their boundaries and $L$ is the union of the $1$-boundaries $\cup \partial_1 \mathcal B_i$. Observe that $L \cap \partial N$ is a pants decomposition of $\partial N$. We also have the following corollary of \Cref{p:agol-blocks}. Note that the nonstandard path metric we placed on the pants graph at the beginning of \Cref{sec:lower-bound} is precisely to accommodate the fact that each $\mathfrak B^S$ block contributes twice as much volume as a $\mathfrak B^T$ block.

\begin{corollary}
    $N - L$ admits a convex hyperbolic metric with totally geodesic thrice-punctured sphere boundary components and volume $\Vol(N - L) = V_{oct}(n_T + 2n_S)$ where $n_T$ is the number of $\mathcal B^T$ blocks and $n_S$ is the number of $\mathcal B^S$ blocks.
\end{corollary}

This result, which is stated as Corollary~4.2 in \cite{EndPeriodic1}, hints at where we are going with all of this. If we can build a model manifold $\widehat M_f$ (for $\overline M_f$) out of pants blocks corresponding to a path in the pants graph so that the volume of $\widehat M_f$ bounds the volume of $\overline M_f$, then we are well on our way to proving the upper bound.

\subsection{Building our model manifold}

Our main strategy for building the model manifold $\widehat M_f$ is to first attempt to decompose $\overline M_f$ into pants blocks and then collapse everything outside of these blocks. In other words, we will obtain $\widehat M_f$ (which admits a block decomposition) as a quotient of $\overline M_f$.

So lets begin with our decomposition of $\overline M_f$ into pants blocks. First choose a pants decomposition $P$ in an $f$-invariant component $\Omega \subset \mathcal P(S)$ and a path $P = P_0, P_1, \ldots, P_n = f^{-1}(P)$ in $\Omega$. We will realize this path via pants blocks in $M_f$ as shown in \Cref{fig:PantsMovesMappingTorus}. Define $L = \cup_{k} P_k \subset M_f$.

\begin{figure}
    \centering
    \includegraphics[width = \textwidth]{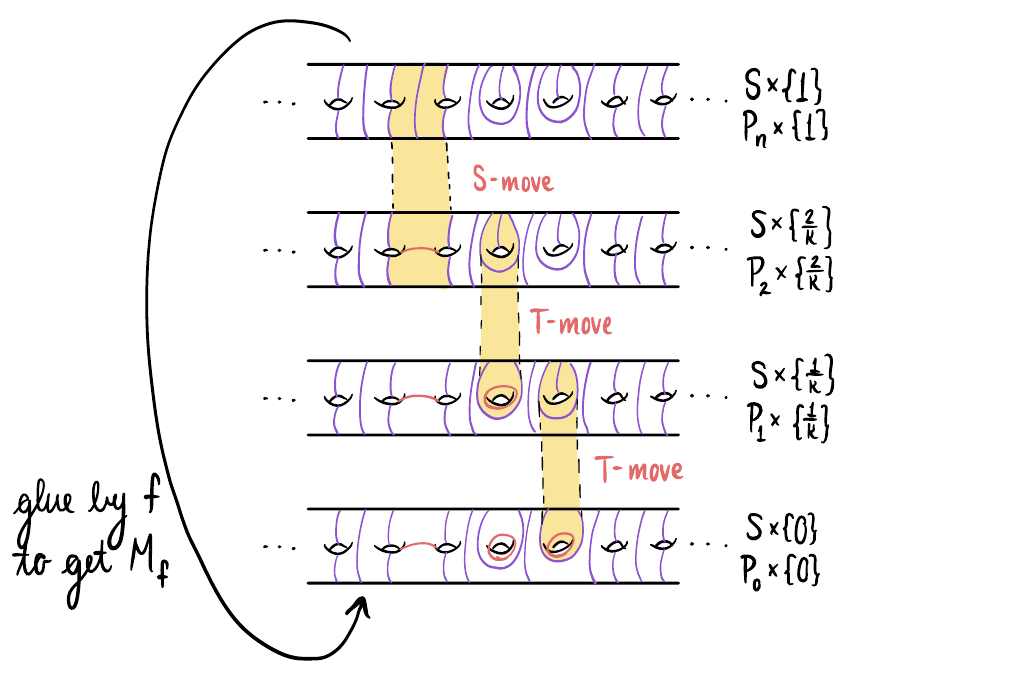}
    \caption{Realizing the path $P = P_0, P_1, \ldots, P_n = f^{-1}(P)$ in $M_F$. Recall that S-moves and T-moves are the elementary moves associated to edges in the pants graph.}
    \label{fig:PantsMovesMappingTorus}
\end{figure}

In particular, we are able to show that there exists a quotient map $h: \overline M_f \to \widehat M_f$ satisfying the following:

\begin{enumerate}
    \item[1)] $\widehat L = h(L)$ is a link.
    \item[2)] The image of the yellow pieces shown in \Cref{fig:PantsMovesMappingTorus} under $h$ form a block decomposition of $(\widehat M_f, \widehat L)$.
    \item[3)] $P_{\Omega} = h^{-1}(\widehat L \cap \partial \widehat M_f) \subset \partial \overline M_f$ is a pants decomposition depending only on $\Omega$, up to isotopy.
    \item[4)] $\overline M_f - P_{\Omega}$ is acylindrical.
\end{enumerate}

This is proved as Proposition 4.3 in \cite{EndPeriodic1}.

Since $\overline M_f - P_{\Omega}$ is acylindrical, it admits a convex hyperbolic structure with totally geodesic, thrice punctured sphere boundary components. 

\smallskip

\noindent \underline{Idea of proof:} Time to flow!

\smallskip

There is a suspension flow $(\psi_s)$ on $M_f$ induced by the upward flow $(x, t) \mapsto (x, t + s)$ on $S \times \mathbb R$. This can be reparametrized to a local flow $\Psi_s$ of $\overline M_f$ so that the flow lines that exit every compact set limit to a point on $\partial \overline M_f$ in finite time. Note that, due to irreducibility of $f$, every flow line of $\Psi_s$ hits at least one of the yellow regions shown in \Cref{fig:PantsMovesMappingTorus} (these correspond to pants blocks).

So we can decompose $\overline M_f$ into either singletons contained in a yellow region (corresponding to a pants block) or maximal components of flow lines of $\Psi_s$ contained in the complement of the yellow regions. These compact arcs get collapsed to points under $h$ as shown in \Cref{fig:BlockQuotient}. 

\begin{figure}
    \centering
    \includegraphics[width = .5\textwidth]{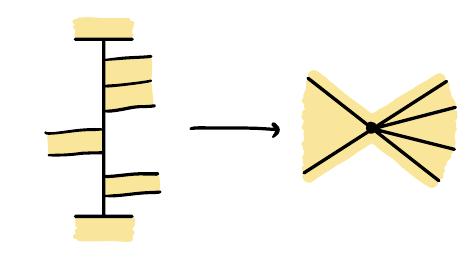}
    \caption{A cartoon illustrating the collapsing of flow lines of $\Psi_s$ under $h$.}
    \label{fig:BlockQuotient}
\end{figure}

The following result of Armentrout is crucial to our proof.

\begin{theorem}[\cite{Arm69}] \label{thm:Armentrout}
    Suppose $M$ is a $3$-manifold and $G$ is a cellular decomposition of $M$ such that $M/G$ is a $3$-manifold. Then the quotient map $M \to M/G$ is homotopic to a homeomorphism.
\end{theorem}

Although we will not prove it here, a good deal of work is needed to show that $\widehat M_f$ is indeed a $3$-manifold so that we can apply \Cref{thm:Armentrout}. See Section~5 of \cite{EndPeriodic1}.

\subsection{Dehn filling and volume}

One of the key tools that we will need to complete our proof of the upper bound is \emph{Dehn filling}. In particular, we will use a process often referred to as \emph{drilling and filling}. Given a $3$-manifold $N$ we can remove the open tubular neighborhood of a simple closed curve $\gamma$ to obtain a torus boundary component $T$ (note that this can be thought of as drilling $\gamma$ out of $N$). We can then \emph{Dehn fill} $N - \gamma$ along $T$ by gluing in a solid torus $D \times S^1$ using a homeomorphism $f: D \times S^1 \to T$ as shown in \Cref{fig:DehnFilling}.

\begin{figure}[ht]
    \centering
    \includegraphics[width = .5\textwidth]{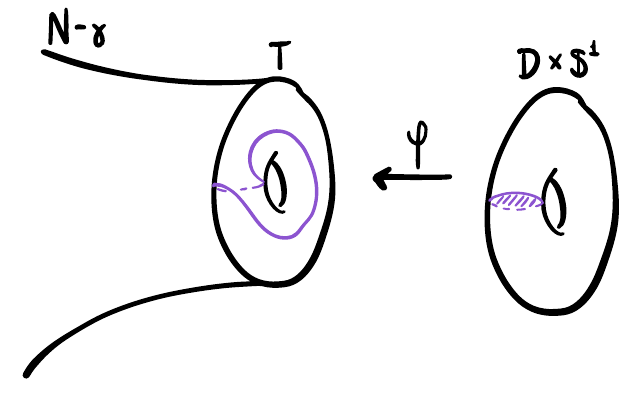}
    \caption{An illustration of gluing a solid torus $D \times S^1$ into the torus boundary component $T$ of $N - \gamma$ with a specified choice of Dehn filling coefficient.}
    \label{fig:DehnFilling}
\end{figure}

The homeomorphism type of this gluing map $f$ is determined by the isotopy class of simple closed curve $\beta \subset T$ to which $\partial D \times {x}$ is identified. We call $\beta$ the \emph{Dehn filling coefficient}. A choice of coefficient is illustrated by the purple curve on T in \Cref{fig:DehnFilling}. Since $\beta$ is a simple closed curve on $T$, then, after choosing a basis for $\pi_1(T) \cong \mathbb Z^2$, it can be described by a pair of of relatively prime integers $(p, q)$. We can the denote the filled manifold by $(N - \gamma)(\beta) = (N - \gamma)(p,q)$.

This process of drilling and filling can be extended to a link $L \subset N$, which is the image of an embedding of a disjoint union of circles. We will make use of the following result which addresses how Dehn filling changes the monodromy for a fibered $3$-manifold. 

\begin{proposition}[\cite{StallingsFiber}]
    Let $\gamma$ be a curve in $Y$ and $K$ the associated knot in the mapping torus $M_f$ for $f \in \Map(Y)$. Then $(M_f - K)(1, n) \cong M_{f \circ T_{\gamma}^n}$. More generally, if $L$ is a link consisting of curves $\gamma_1, \ldots, \gamma_k$ with $\gamma_i$ lying on $Y \times \{\frac{i}{k+1}\} \subset Y \times (0,1) \subset M_f$, for each $i$, then \[(M_f - L)((1, n_1), \ldots, (1,n_k)) \cong M_{f \circ T_{\gamma_k}^{n_k} \circ \cdot \circ T_{\gamma_1}^{n_1}}.\]
\end{proposition}

So how does this procedure of drilling and filling relate to volume? By utilizing Thurston's Hyperbolic Dehn Surgery Theorem which we state below.

\begin{theorem}
    Let $(N, \sigma)$ be a finite volume hyperbolic $3$-manifold with totally geodesic boundary and $k$ torus cusps. Suppose $\beta^n = (\beta_1^n, \ldots, \beta_k^n)$ are a sequence of Dehn filling coefficients so that for each $i$, either $\beta_i^n = \infty$ for all $n$, or $\beta_i^n \to \infty$. Then for $n$ sufficiently large, $N(\beta^n)$ admits a complete hyperbolic structure $\sigma(\beta^n)$ with totally geodesic boundary and \[\Vol(N(\beta^n), \sigma(\beta^n)) \to \Vol(N, \sigma).\] Moreover, \[\Vol(N(\beta^n), \sigma(\beta^n)) \leq(N, \sigma)\] with equality if and only if $\beta^n = (\infty, \ldots, \infty).$     
\end{theorem}

\subsection{Putting the pieces together}

Finally, we are ready to give a sketch of the proof of the upper bound, which we will do in two steps.

\begin{proof}[Step 1:]
    Show that $\Vol(\overline M_f - P_{\Omega}) \leq V_{oct} \tau(f, \Omega)$. 

    \medskip

    Fix $\Omega \subset \mathcal P(S)$ and let $\varepsilon > 0$. Pass to a power $N > 0$ with $P = P_0, \ldots, P_n = f^{-N}(P)$ so that $|\tau(f, \Omega) - \frac{n_T + 2n_S}{N}| < \varepsilon,$ where $n = n_T + n_S$. Note that $P = P_0, \ldots, P_n = f^{-N}(P)$ is built from $P = P_0', \ldots, P_k' = f^{-1}(P)$ and its images under $f^{-1}, \ldots, f^{-N+1}$.

    Thus, we obtain a pants decomposition $P_{\Omega}^N$ of $\partial \overline M_{f^N}$ where $P_{\Omega}^N$ is a lift of the pants decomposition $P_{\Omega}$ of $\partial \overline M_f$. So $\Vol(\overline M_{f^N} - P_{\Omega}^N) = N \cdot \Vol(\overline M_f - P_{\Omega})$ and we are only left to prove the following claim. 

    \begin{claim}
        $\Vol(\overline M_{f^N} - P_{\Omega}^N) \leq V_{oct}(n_T + 2n_S)$.
    \end{claim}
 
    Assuming the claim, we have that $\Vol(\overline M_f - P_{\Omega}) = \frac{\Vol(\overline M_{f^N} - P_{\Omega}^N)}{N} \leq V_{oct} \frac{n_T + 2n_S}{N} \leq V_{oct}(\tau(f, \Omega) + \varepsilon)$. Thus, letting $\varepsilon \to 0$ we are done (modulo the proof of the claim) with Step 1!
\end{proof} 

\begin{proof}[Proof of Claim:] Recall that we have a pants decomposition of $(\widehat M_{f^N}, \widehat L)$ and thus, $\Vol(\widehat M_{f^N} - \widehat L) = V_{oct}(n_T + 2n_S)$. By performing Dehn filling on $\widehat M_{f^N} - \widehat L$, we obtain $\overline M_{f^N} - P_{\Omega}^N$. Thus, by Thurston's Dehn Surgery Theorem $\Vol(\overline M_{f^N} - P_{\Omega}^N) \leq \Vol(\widehat M_{f^N} - \widehat L) = V_{oct}(n_T + 2n_S)$, as desired.  
\end{proof}

\begin{proof}[Step 2:]
    Show that $\Vol(\overline M_f - P_{\Omega}) \leq V_{oct} \cdot \tau(f, \Omega)$ implies that $\Vol(\overline M_f) \leq V_{oct} \cdot \tau(f)$.

    \medskip

    Given $\varepsilon > 0$, suppose $\Omega \subset \mathcal P(S)$ such that $|\tau(f) - \tau(f, \Omega)| < \varepsilon$ and let $N> 0$ such that $f^N$ preserves $\Omega$. Then $\frac{\tau(f^N, \Omega)}{N} = \tau(f, \Omega) \leq \tau(f) - \varepsilon$. 

    By Step 1, $\Vol(\overline M_{f^N}) \leq \Vol(\overline M_{f^N} - P_{\Omega}) \leq V_{oct} \cdot \tau(f^N, \Omega).$ Thus, $\Vol(\overline M_f) = \frac{\Vol(\overline M_{f^N})}{N} \leq \frac{V_{oct}\cdot \tau(f^N, \Omega)}{N} \leq V_{oct} \cdot \tau(f) + \varepsilon$. Letting $\varepsilon \to 0$ we are done.
\end{proof} 

\section{What's next?}

We have made some solid progress on \Cref{q:taut-geometry}, which we recall below for the reader's convenience, but there is lots more to do! Since this paper aims to be expository rather than a broad survey, we will only mention a handful of possible directions here. We encourage the interested reader to explore further. 

\TautGeometry*

For example, in recent work \cite{Whitfield}, Whitfield generalizes the following result of Minsky \cite[Theorem~B]{Minsky.2000} to the setting of end-periodic mapping tori.

\begin{theorem}[Minsky]
    Given a surface $S$, $\varepsilon > 0$, and $L > 0$, there exists $K$ so that, if $\rho: \pi_1(S) \to \textrm{PSL}_2(\mathbb C)$ is a Kleinian surface group and $Y$ is a proper essential subsurface of $S$, then \[\diam_Y(C_0(\rho, L)) > K \Rightarrow \ell_{\rho}(\partial Y) \leq \varepsilon,\] where $C_0(\rho, L)$ is the set of homotopy classes of simple closed curves $\alpha$ such that $\ell_{\rho}(\alpha) \leq L.$
\end{theorem}

In particular, they prove the following where $\lambda^-$ and $\lambda^+$ are the \emph{negative} and \emph{positive Handell--Miller laminations} of $S$ associated to the end-periodic homeomorphism $f:S \to S$. For a comprehensive treatment of the construction and structure of these laminations see \cite{CC-book}.

\begin{theorem}[Whitfield]
    For any $D, \varepsilon > 0,$ there exists $K = K(D, \varepsilon)$ such that for any atoroidal end-periodic homeomorphism $f:S \to S$ with $\chi(f) + \xi(f)^2 \leq D,$ there exists $s \in \textrm{AH}(M_f)$ satisfying the following: For any connected, compact subsurface $Y \subset S \subset M_f$, we have \[d_{Y}(\lambda^+, \lambda^-) \geq K \Rightarrow \ell_{s}(\partial Y) \leq \varepsilon.\]
\end{theorem}


In a similar spirit, we could hope to produce uniform biLipshitz models for end-periodic mapping tori. Such model manifolds were key in the proof of the Ending Lamination Conjecture \cite{ELC1, ELC2}. Furthermore, such models would be a significant step, together with forthcoming work of \cite{BKM}, towards building models for arbitrary closed $3$-manifolds admitting taut depth-one foliations using the combinatorics of the foliations. However, a more approachable goal would be to simply give volume bounds on arbitrary closed $3$-manifolds admitting taut depth-one foliations using the work of \cite{EndPeriodic1}, \cite{EndPeriodic2}, and \cite{BKM}.

We could also approach \Cref{q:taut-geometry} from a completely different perspective, pseudo-Anosov flows! This is the perspective taken in a series of works by Landry, Minsky, and Taylor \cite{LandryMinskyTaylor2021, LandryMinskyTaylor2023}. They show that any end-periodic map can be obtained from a fibration using a technique called ``spinning" (this is different than the spinning illustrated in \Cref{fig:Spinning} although very similar in spirit). The beauty of their work is that they are able to produce a standard representative (a so-called \emph{spun pseudo-Anosov}) of any (atoroidal) end-periodic homeomorphism that is directly related to a pseudo-Anosov representative of the return map to a compact fiber. This perspective opens up a whole new possibility of questions and approaches which in itself could constitute an entire paper.

\bibliographystyle{alpha}
  \bibliography{main}
  
\end{document}